\documentclass[12pt,a4paper]{article} 
\usepackage{amsmath, latexsym, amsthm, amsfonts, amssymb}
\usepackage{hyperref}
\usepackage[shortlabels]{enumitem}

\usepackage{chngcntr}
\counterwithin{equation}{section}

\newcommand{\R}{\ensuremath{\mathbb{R}}}

\newtheorem{theorem}{Theorem}[section]
\newtheorem*{theorem*}{Theorem}

\newtheorem*{prop*}{Proposition}
\newtheorem{prop}[theorem]{Proposition}
\newtheorem{remark}[theorem]{Remark}

\newtheorem{lemma}[theorem]{Lemma}
\newtheorem{defn}[theorem]{Definition}
\newtheorem{exmp}[theorem]{Example}
\newtheorem{corollary}[theorem]{Corollary}
\newtheorem{labelledcorollary}[theorem]{Corollary}

\DeclareMathOperator\arctanh{arctanh}
\DeclareMathOperator{\csch}{csch}
\DeclareMathOperator{\loc}{loc}

\title{\vspace{-1.0cm}Existence and Uniqueness for the Non-Compact Yamabe Problem  of Negative Curvature Type}
\author{Joseph Hogg~ and Luc Nguyen~\footnote{Mathematical Institute and St Edmund Hall, University of Oxford, Andrew Wiles
		Building, Radcliffe Observatory Quarter, Woodstock Road, Oxford OX2 6GG, UK. Email:
		luc.nguyen@maths.ox.ac.uk}}
\date{}

\begin{document}
    \maketitle
    \vspace{-1.0cm}
    \begin{abstract}
        We study existence and uniqueness results for the Yamabe problem on non-compact manifolds of negative curvature type. Our first existence and uniqueness result concerns those such manifolds which are asymptotically locally hyperbolic. In this context, our result requires only a partial $C^2$ decay of the metric, namely the full decay of the metric in $C^1$ and the decay of the scalar curvature. In particular, no decay of the Ricci curvature is assumed. In our second result we establish that a local volume ratio condition, when combined with negativity of the scalar curvature at infinity, is sufficient for existence of a solution. Our volume ratio condition appears tight. This paper is based on the DPhil thesis of the first author.
    \end{abstract}
    \vspace{-0.5cm}
    \tableofcontents   
    \section{Introduction}

    We are interested in the Yamabe problem on non-compact manifolds: given some complete non-compact Riemannian manifold $(M, g)$ of dimension $n\geq 3$, does there exist a corresponding complete conformal metric whose scalar curvature is constant? Equivalently, we would like to find a complete metric $\tilde{g} = u^\frac{4}{n-2}g$ where $u$ is some strictly positive smooth function on $M$ solving the Yamabe equation
    \begin{equation*}\label{scalchange}
    	S_{\tilde{g}} = u^{-\frac{n+2}{n-2}}\left(-c_n\Delta_g u + S_g u \right) \equiv \text{constant,} \qquad \text{ $c_n := \frac{4(n-1)}{n-2}$}.
    \end{equation*}
    Here, $S_g$ and $S_{\tilde{g}}$ refer to the scalar curvatures of the corresponding metrics. The operator $-c_n\Delta_g +S_g$ is known as the conformal Laplacian.
    
    In the case that $M$ is compact, the Yamabe problem has been extensively studied. The existence of a solution was established through the combined works of Yamabe \cite{MR125546}, Trudinger \cite{MR240748}, Aubin \cite{ MR431287} and Schoen \cite{MR788292}. For other aspects of the Yamabe problem, see e.g. \cite{MR2425176, MR2472174, MR982351,  MR1679784, MR544879, MR634248, MR2477893,  MR1666838,  MR2057491, MR2164927, MR2309836, MR2393072, MR4436213} and references therein.
    
    Our work is focused on the Yamabe problem of ``negative curvature type" on non-compact manifolds, namely we are interested in obtaining conformal changes to constant negative scalar curvature. Consequently, asserting that $S_{\tilde{g}} \equiv -n(n-1)$, the equation for the conformal change in scalar curvature yields the Yamabe equation
    \begin{equation}\label{yaeqintro}
    -c_n\Delta_g u + S_g u = -n(n-1)u^\frac{n+2}{n-2} \ . \tag{Ya}
    \end{equation}
    Finding a solution to the Yamabe problem thus amounts to finding a positive solution of the equation \eqref{yaeqintro} above for which the corresponding conformal metric is complete.
    
    The Yamabe problem of negative curvature type on non-compact manifolds has been studied extensively in the literature. Important progress has been made by Andersson, Chru\'{s}ciel and Friedrich \cite{MR1186044}, Gover and Waldron \cite{MR3668619}, Graham \cite{grahampaper} and Loewner and Nirenberg \cite{MR0358078}. For further literature see e.g. \cite{MR660747,MR816672, MR932852,  aviles1988,MR2413198,MR1641725, MR1760721,  MR4096721, han2019singular,  MR4269604, MR1032773, 2005Labutin,MR0358078, MR1142715, MR1425579, MR1465899, MR931204,  MR626622} and references therein.

    In the present work we consider conditions for existence of a solution to the Yamabe problem on a given non-compact Riemannian manifold $(M, g)$ where $g$ has asymptotically negative scalar curvature. We assume throughout the paper that $g$ is complete and satisfies a condition of the type
    \begin{equation}\label{asymneg}
    	\limsup S_{g} \leq -\varepsilon < 0
    \end{equation}
    for some $\varepsilon>0$ and where the limit is taken along any divergent sequence in the manifold.
    
   It is known that \eqref{asymneg} is insufficient to conclude that the Yamabe problem can be solved. For example, in \cite{aviles1988}, alongside a number of existence results, Aviles and McOwen give an example of a complete metric $g$ on $\R\times \mathbb{T}^{n-1}$ satisfying \eqref{asymneg} for which the Yamabe problem has no solution. It is therefore of interest to understand what conditions, in addition to \eqref{asymneg}, are necessary and/or sufficient for the Yamabe problem to have a solution. 
    
    Our first result concerns asymptotically locally hyperbolic (ALH) manifolds, a well-studied class of manifolds satisfying \eqref{asymneg}, in a weaker sense than considered in the existing literature (e.g. \cite{MR1186044} and \cite{MR3761652}). In particular, our notion of ALH requires only $C^1$ decay in the metric components to those of the model space and a bound on the scalar curvature of the form
    \begin{align}
    S_g \leq -n(n-1) + C&e^{-\alpha r} \qquad \text{ or }\label{introupperdecay}\\	
	|S_g + n(n-1)| \leq C&e^{-\alpha r}\label{introfulldecay}
    \end{align}
    for $\alpha \in (0,n)$ and some constant $C>0$, without requiring decay of the full curvature tensor. For the precise definition, see Section \ref{defs}. We establish:
    \begin{theorem}\label{thmA}
        Suppose $(M, g)$ is an ALH manifold of regularity order $1$ and with decay exponent $\alpha \in (0,n)$. If the scalar curvature satisfies \eqref{introupperdecay} on $M$, then there exists a positive smooth solution $u$ to \eqref{yaeqintro} on $M$ satisfying $\liminf_{r\rightarrow \infty} u \geq 1$ corresponding to a complete conformal metric $\tilde{g}$ with $S_{\tilde{g}} \equiv -n(n-1)$. 
        
        If the scalar curvature satisfies the stronger condition \eqref{introfulldecay} on $M$, then the solution $u$ above is the unique solution satisfying $\lim_{r\rightarrow \infty} u = 1$, the corresponding conformal manifold $(M, \tilde{g})$ is ALH of regularity order $0$ and with decay exponent $\alpha$, and $u$ is maximal among all solutions to \eqref{yaeqintro}, that is any other solution $\tilde{u}$ satisfies $\tilde{u} \leq u$.
    \end{theorem}
	
	Some comments on the conclusions of the theorem are in order. Concerning the uniqueness, we are able to slightly relax the condition $\lim_{r\rightarrow\infty} u = 1$; see Proposition \ref{uniqprop}. However, this necessary condition for uniqueness cannot be dropped in its entirety:
	
	\begin{exmp}\label{condneededintro}
		 Take $M = B_1$ to be the unit ball in $\R^n$, $\delta$ to be the Euclidean metric and $$g = \frac{4}{(1-|x|^2)^2} \delta \quad \text{ and } \quad g'= \frac{4R^2}{(R^2-|x|^2)^2} \delta$$ to be the Poincar\'{e} metrics on $B_1$ and $B_R$ for some $R>1$. Clearly, on $B_1$, if we write $g' = u^\frac{4}{n-2}g$ then $u$ must solve the Yamabe equation for $(M = B_1, g)$ and $\lim_{r(x)\rightarrow\infty} u (x) = 0$. In fact, $u(x) e^{\frac{n-2}{2}r(x)}$ tends to $\left(\frac{4R}{R^2 - 1}\right)^\frac{n-2}{2}$ as $r(x)\rightarrow\infty$. 
	\end{exmp}
	
	Regarding the loss in regularity order of the resulting conformal ALH metric, this is expected due to the assumption that the scalar curvature decays only in $L^\infty$ in \eqref{introfulldecay}. Under stronger decay assumptions, one may appeal to elliptic regularity theory to obtain correspondingly stronger regularity of the resulting conformal metric at infinity.
	
	We place our result in the context of the existing literature on the Yamabe problem for asymptotically hyperbolic manifolds. A fundamental and pioneering work of Loewner and Nirenberg \cite{MR0358078} showed that, on every open subset $M$ of $\R^n$ with regular boundary, there exists a complete and conformally flat metric $g$ which is ALH and has constant negative scalar curvature. Of special relevance to Theorem \ref{thmA} is the work of Andersson, Chru\'{s}ciel and Friedrich \cite{MR1186044} which proved, roughly speaking and among other things, the existence and uniqueness of a solution to the Yamabe problem when $(M,g)$ is ALH and admits a $C^2$ conformal compactification\footnote{The related question of when intrinsic definitions of non-compact ALH manifolds imply the existence of a corresponding conformal compactification of a certain regularity is addressed, for example, in \cite{MR2836592} and \cite{MR3169748}.}. We note that the existence of such a conformal compactification mandates the full decay of all sectional curvatures to a negative constant near infinity. Under weaker regularity assumptions on the conformal compactification, while still maintaining the full curvature decay, a similar result has been obtained in \cite{MR3761652}. We remark that in Theorem \ref{thmA}, we require instead an assumption on the scalar curvature alone. It is not hard to give examples of metrics which satisfy the conditions of Theorem \ref{thmA} but whose Ricci curvature is not asymptotic to a constant multiple of the metric; see Example \ref{3dscexample}. 

	In our second main result, we give a local condition which, together with \eqref{asymneg}, ensures the existence of a solution to the Yamabe problem. Our result is motivated by a result of Aviles and McOwen in \cite{aviles1988} which asserts that the negativity of the first eigenvalue of the conformal Laplacian on some compact domain, in combination with \eqref{asymneg}, implies the existence of a solution to the Yamabe problem. We prove:
	
	\begin{theorem}\label{dcexistthm}
		Let $(M,g)$ be a Riemannian manifold and suppose that there exist two open sets $\Omega_1\subset \Omega_2$ with $C^1$ boundary which satisfy (taking $d_g$ to be the geodesic distance with respect to $g$) that 
		\begin{equation}\label{eqspace}
			d_g(x,\partial \Omega_2) = R \text{ for each }x\in\partial\Omega_1 \text{ for some } R>0,
		\end{equation} 
		\begin{equation}\label{dccondexist}
		\frac{Vol_g(\Omega_2\setminus \Omega_1)}{Vol_g(\Omega_1)} \leq \sinh^2\left(\frac{\sqrt{n(n-2)}}{2}R\right),
		\end{equation}
		and the scalar curvature satisfies $S_g \leq -n(n-1)$ on $\Omega_2$. Then, the conformal Laplacian $-c_n\Delta_g + S_g$ for $(M, g)$ has a negative first eigenvalue on $\Omega_2$. Consequently, if $g$ satisfies \eqref{asymneg}, there exists a complete metric $\tilde{g}$ conformal to $g$ on $M$ of constant scalar curvature $-n(n-1)$.
	\end{theorem}

	Our Theorem \ref{dcexistthm} gives a new sufficient condition for existence, relating a local volume growth to the solvability of the Yamabe problem on the whole manifold. The local condition \eqref{dccondexist} is tight in the following sense:
	
	\begin{remark}\label{sharpremark}
		The constant $\frac{\sqrt{n(n-2)}}{2}$ in \eqref{dccondexist} is sharp with respect to the existence of a negative first eigenvalue for the conformal Laplacian. Namely, for any $\beta >  \sqrt{n(n-2)} $ there exists a manifold $(M, g)$ and concentric annuli $\Omega_1 \subset \Omega_2$ satisfying \eqref{eqspace} with volume ratio
		\begin{equation}
		\frac{Vol_g(\Omega_2\setminus \Omega_1)}{Vol_g(\Omega_1)} \leq \sinh^2\Big(\frac{\beta}{2} R\Big)
		\label{Eq:VRSinh}
		\end{equation} and with scalar curvature satisfying $S_g \leq -n(n-1)$ on $\Omega_2$ but for which the first eigenvalue for the conformal Laplacian on $\Omega_2$ is positive. See Proposition \ref{sharpprop} and the comment following it.
	\end{remark}
	
	 While the relationship between volume comparison and the first eigenvalue for the Laplacian $-\Delta_g$ has been well observed in the literature, the relation between volume comparison and the solvability of the Yamabe problem is less understood. Although it is sharp for the existence of a negative first eigenvalue for the conformal Laplacian, we do not know whether our bound in \eqref{dccondexist} is optimal for the solvability of the Yamabe problem. It would be interesting to understand what the sharp volume ratio bound of the form \eqref{dccondexist} would be for existence of solutions to the Yamabe problem.
	
	Theorem \ref{dcexistthm} is of a different flavour to Theorem \ref{thmA}. We first note that, for the ALH manifolds treated in Theorem \ref{thmA}, the volume ratio in the asymptotic region behaves differently than \eqref{dccondexist}; in particular if we fix $\Omega_1$ in the asymptotic region and consider large $R$, the volume ratio satisfies \begin{equation}\label{lhexmp}
		\frac{Vol_g(\Omega_2\setminus \Omega_1)}{Vol_g(\Omega_1)} \approx e^{(n-1)R} \gg \sinh^2\left(\frac{\sqrt{n(n-2)}}{2}R\right).
	\end{equation}
	Furthermore, we have the following remark:
	\begin{remark}
		Consider a manifold $(M, g)$ of the form $M = \R \times N$ where $(N, h)$ is a manifold of constant negative scalar curvature and $g = dr^2 + e^{-\beta r}h \text{ for some } 0\leq\beta<\frac{\sqrt{n(n-2)}}{n-1}.$
		By \eqref{scal}, a computation gives $S_g = -(n-1)(n-2)e^{2\beta r} + \mathcal{O}(1)\rightarrow-\infty$ as $r\rightarrow\infty$. Additionally, taking $\Omega_1 = (r_0, r_0 + 1)\times N$, $\Omega_2 = (r_0 - R, r_0 + 1 + R)\times N$  for some $r_0 \gg R \gg 1$ we have that \eqref{dccondexist} is satisfied. And lastly, by Proposition \ref{wpeconfprop}, $g$ is not conformal to a locally hyperbolic metric. Consequently, $(M, g)$ provides an example which is not covered by Theorem \ref{thmA} and for which we may solve the Yamabe problem on $(M, g)$.
	\end{remark}
	
	We briefly discuss the proofs of our results. The proof of the existence part of Theorem \ref{thmA} makes use of the procedures in \cite{MR1186044} and \cite{aviles1988} where barrier functions play an important role. Due to the weaker asymptotic conditions on the scalar curvature in \eqref{introupperdecay}, or \eqref{introfulldecay}, additional work is required to construct such barrier functions. For the uniqueness part, we treat separately the lower and upper bounds at infinity for solutions of \eqref{yaeqintro}. For the upper bound, we adapt certain ideas from \cite{MR0358078} using certain reference solutions. The main difference is that we lack an explicit closed form of the reference solutions and so the final conclusion is drawn from an asymptotic analysis of these solutions, exploiting the conformal invariance of the equation. For the lower bound, we highlight that our treatment is very different from approaches seen in the literature, where one typically uses the existence of a sufficiently regular conformal compactification in order to compare the solution to some reference solution exterior to the compactified manifold. Owing to our lack of such a compactification, we instead perform a blow-up analysis intrinsically to the given manifold, identifying a limiting equation from which we deduce the desired lower bound at infinity.
	
	To prove Theorem \ref{dcexistthm}, we make use of a particular test function for the variational functional corresponding to the first eigenvalue of the conformal Laplacian. This test function is constructed by minimising the sup-norm of the integrand in the variational functional, following the procedure in \cite{MR196551} and using the volume ratio condition \eqref{dccondexist}. It was surprising to the authors that condition \eqref{dccondexist} for the volume ratio turned out to be sharp in the sense of Remark \ref{sharpremark} above. While falling outside the scope of the present paper, it would be interesting to understand what other geometric quantities similar in nature to the volume ratio used here could have relevance to the behaviour of the first eigenvalue of the conformal Laplacian and the solvability of the Yamabe problem.
	
	The paper is structured as follows. In Section \ref{alhsec}, we address Theorem \ref{thmA} and our work in the context of ALH manifolds. In Section \ref{dcsec}, we address Theorem \ref{dcexistthm} and our related study of the volume ratios of multiply warped product manifolds.
	
	\subsubsection*{Acknowledgments.} J.H. was partially supported by the EPSRC Centre for
		Doctoral Training in Partial Differential Equations [grant number EP/L015811/1]. The authors would like to thank the anonymous referee for their constructive comments.
		
    \section[The Yamabe problem for ALH manifolds]{The Yamabe problem for asymptotically locally hyperbolic manifolds}\label{alhsec}
	
	In this section, we prove Theorem \ref{thmA}. We first make precise our setting in Section \ref{defs}. Our results are then split into two main sections; Section \ref{existsec} addresses the existence of a solution to the Yamabe problem and Section \ref{asymsec} focuses on an understanding of the asymptotic behaviour of the obtained conformal factor. We then summarise how the results of the previous two sections combine to prove Theorem \ref{thmA} in Section \ref{proofsec}. Finally, in Section \ref{wpsubsec}, we provide a discussion of how Theorem \ref{thmA} can be applied to the broader class of asymptotically warped product manifolds, in particular via a result regarding the conformal classes of reference warped product metrics.

    \subsection{Definitions and notation}\label{defs}

We will consider manifolds which may be decomposed as a union $M= M_0 \cup M^+$, where $M_0$ is some compact interior region, $M^+$ is a non-compact exterior region and both parts are disjoint apart form their common boundary. We assume further that we may express $M^+ = \R_{\geq 0} \times N$ where $N$ is some $(n-1)$ - dimensional compact manifold. On the end $\R_{\geq 0} \times N$, we denote by $r$ a coordinate on the $\R_{\geq 0}$ fibre. Additionally, we denote the coordinates on any local (angular) chart on $N$ with a $\theta^a$, where $a=1,\ldots,n-1$, and we use $a,b,c,\ldots$ to index angular coordinates. When referring to the full set of coordinates on $M^+$ we use the notation $x^1 = \theta^1, \ldots, x^{n-1} = \theta^{n-1}$, $x^n =r$ and we use $i,j,k,\ldots$ to index over all coordinates.

We define a reference locally hyperbolic metric $\mathring{g}$ on the exterior region $M^+$ to be
\begin{equation}\label{referencemetric}
\mathring{g} = dr^2 +f_k^2(r+r_0)\mathring{h}
\end{equation}
for some $r_0>0$ where $\mathring{h}$ is a metric on $N$ of constant scalar curvature $(n-1)(n-2)k$ for $k\in\{-1, 0, 1\}$ and
\begin{equation*}
f_k(r) = \begin{cases}
\sinh(r) & \ k=1,\\
e^r & \  k=0,\\
\cosh(r) & \ k=-1.
\end{cases}
\end{equation*}
In particular, when $k=1$ and $N=\mathbb{S}^{n-1}$, $\mathring{g}$ is the standard hyperbolic metric.

We additionally note that the scalar curvature of a warped product metric like $\mathring{g}$ may be computed via the formula
\begin{equation*}
S_{\mathring{g}} = -2(n-1)\frac{f_k^{\prime\prime}}{f_k} - (n-1)(n-2)\left(\frac{f_k^\prime}{f_k}\right)^2 + \frac{S_{\mathring{h}}}{f_k^2}.
\end{equation*}
From this, one may compute that $S_{\mathring{g}}\equiv -n(n-1)$ for each $k$ in the above definition.

For the next definition below, we choose a finite set of preferred charts $U_i$ covering $N$, each with a preferred choice of local coordinates $\{\theta^1,\ldots, \theta^{n-1}\}$. We extend these charts to $M^+$ by defining $V_i = \R_{\geq 0}\times U_i$ with coordinates $\{r,\theta^1,\ldots,\theta^{n-1}\}$ and fix them from hereon. For a function $\varphi$ on the end $\R_{\geq 0}\times N$ we use the notation $\varphi = \mathcal{O}_m(e^{-\alpha r})$ to indicate that $\varphi$ and all of its first $m$ derivatives in the coordinates defined above decay as $e^{-\alpha r}$, that is there exists a constant $C>0$ such that $\varphi$ satisfies $|\partial_\beta \varphi| \leq Ce^{-\alpha r}$ where $\beta$ indicates any multi-index with $0 \leq |\beta|\leq m$.

We adopt the following definition of ALH manifolds which, while differing slightly, is consistent with those found in the literature (e.g. in \cite{MR3761652,MR2038048, MR3801943, MR1879228}).
\begin{defn}[Asymptotically Locally Hyperbolic]\label{alhdef}
	
	We say a Riemannian manifold $(M, g)$ is \emph{asymptotically locally hyperbolic of regularity order $m\in \mathbb{N}$ and with decay exponent $\alpha > 0$} if we can write $M= M_0 \cup (\R_{\geq 0} \times N)$ and we can write the metric components of $g$ on any preferred coordinate chart as
	\begin{align*}
		g_{rr} = \mathring{g}_{rr} = 1, \quad
		g_{ab} = \mathring{g}_{ab} + \mathcal{O}_m\left(e^{-(\alpha - 2)r}\right), \quad
		g_{ra} = \mathcal{O}_m\left(e^{-(\alpha - 1)r}\right)
	\end{align*} 
	where $\mathring{g}$ is a reference locally hyperbolic metric of the form \eqref{referencemetric} for some $k\in\{-1, 0, 1\}$. If $(N,\mathring{h})$ is the round sphere then we simply say that $g$ is asymptotically hyperbolic.
\end{defn}
It is clear that if $(M, g)$ is ALH with decay exponent $\alpha$ then it is also ALH with decay exponent $\alpha^\prime$ for any $0<\alpha^\prime <\alpha$. We note that if $\alpha < 1$, then $(M, g)$ is $C^{0, \alpha}$ conformally compactifiable and if $\alpha \geq 1$ then $(M, g)$ is $C^1$ conformally compactifiable. For terminology, see for example \cite{MR3801943}.

In the remainder of this paper, we will make regular use of the coordinate function $r$ corresponding to the $\R_{\geq 0}$ fibre of the exterior region $M^+$. We note that the particular choice of $r$ is not unique, in that the reference metric $\mathring{g}$ defined above may be expressed in the form \eqref{referencemetric} for arbitrarily many choices of coordinate function $r$ via diffeomorphism of $M^+$ or by an altogether different choice of splitting of $M$ into the interior and exterior regions $M_0$ and $M^+$. To avoid this complication, whenever we speak of an ALH manifold as defined above, we implicitly assume that there is a pre-chosen $r$.

We provide the following computational lemma establishing the corresponding decay of the metric inverse and Christoffel symbols which will be useful to us later.

\begin{lemma}\label{scalcurvlemma}
	Suppose $(M, g)$ is an ALH manifold of regularity order $1$ and decay exponent $\alpha$. Then, in any of the preferred charts
	\begin{align*}
	\begin{split}
	g^{rr} &= 1 + \mathcal{O}_1(e^{-2\alpha r}), \\
	\Gamma^r_{rr} &= \mathcal{O}\left(e^{-\alpha r}\right),\\
	\Gamma^a_{br} &= \mathring{\Gamma}^a_{br} +  \mathcal{O}\left(e^{-\alpha r}\right),
	\end{split}
	\qquad
	\begin{split}
	g^{ra} &= \mathcal{O}_1(e^{-(\alpha+1) r}),\\
	\Gamma^a_{rr} &= \mathcal{O}\left(e^{-(\alpha+1) r}\right),\\
	\Gamma^r_{ab} &= \mathring{\Gamma}^r_{ab} +  \mathcal{O}\left(e^{-(\alpha-2) r}\right),
	\end{split}
	\qquad
	\begin{split}
	g^{ab} &= \mathring{g}^{ab} + \mathcal{O}_1(e^{-(\alpha+2) r}),\\
	\Gamma^r_{ar} &= \mathcal{O}\left(e^{-(\alpha-1) r}\right),\\
	\Gamma^a_{bc} &= \mathring{\Gamma}^a_{bc} +  \mathcal{O}\left(e^{-\alpha r}\right),
	\end{split}
	\end{align*}
	where we use the notation $\mathring{\Gamma}$ to denote the Christoffel symbols of the reference locally hyperbolic metric $\mathring{g}$ which satisfy $$\mathring{\Gamma}^r_{ab} = \mathcal{O}_2\left(e^{2r}\right),\qquad \mathring{\Gamma}^a_{br} = \mathcal{O}_2\left(1\right), \qquad \mathring{\Gamma}^a_{bc} = \mathcal{O}_2\left(1\right).$$
\end{lemma}

We provide an example of an ALH manifold in the sense of Definition \ref{alhdef} which is not covered by the existing literature for the Yamabe problem on asymptotically hyperbolic manifolds. In particular, the class described below includes metrics where not all of the sectional curvatures are decaying to a negative constant at infinity.

\begin{exmp}\label{3dscexample}
	Consider the warped product manifold $M^3 = \R_{\geq0}\times \mathbb{T}^{2}$ where $\mathbb{T}^{2}$ is the $2$-dimensional flat torus with metric $(dx^1)^2 + (dx^2)^2$ and standard coordinates $\{x^a\}$. We endow $M^3$ with the diagonal metric
	\begin{equation*}
	g = dr^2 + e^{2r}\left(p(r)(dx^1)^2 + p^{-1}(r)(dx^2)^2\right).
	\end{equation*}
	This metric is ALH in the sense of Definition \ref{alhdef} provided, for example, $p(r) = 1 + \mathcal{O}(e^{-\alpha r})$ and $p^\prime(r) = \mathcal{O}(e^{-\alpha r})$, which we assume in this example.
	The metric $g$ has Ricci curvature
	\begin{gather*}
	R_{rr} = -2 - \frac{1}{2}\left(\frac{p^\prime}{p}\right)^2, \qquad R_{11} = e^{2r}p\left(-2 + \frac{1}{2}\left(\frac{p^\prime}{p}\right)^2 - \frac{p^\prime}{p} - \frac{1}{2}\left(\frac{p^{\prime\prime}}{p}\right)\right),\\
	R_{22} = e^{2r}p^{-1}\left(-2 - \frac{1}{2}\left(\frac{p^\prime}{p}\right)^2 + \frac{p^\prime}{p} + \frac{1}{2}\left(\frac{p^{\prime\prime}}{p}\right)\right).
	\end{gather*}
	We see that the Ricci curvature does not necessarily decay to a constant multiple of the metric (note the presence of the $p^{\prime\prime}$ term in $R_{11}$ and $R_{22}$). In contrast, the metric $g$ has scalar curvature
	\begin{equation}
	S_g = -6 - \frac{1}{2}\left(\frac{p^\prime}{p}\right)^2 = -6 - \mathcal{O}(e^{-\alpha r}).
	\end{equation}
\end{exmp}

The example above demonstrates a class which does not satisfy the requirements (discussed in more detail in the following section) in \cite{MR1186044} or \cite{MR3761652} but which falls under Definition \ref{alhdef}. There are certainly many such $p$ which behave wildly in $C^2$ and so have poor behaviour of the Ricci curvature, for example take $p(r) = 1 + e^{-2\alpha r}\sin(e^{\alpha r})$.

\subsection{Existence of a solution for the Yamabe problem}\label{existsec}

In this section we will exhibit a sub-solution to \eqref{yaeqintro}, from which existence of a solution follows via the arguments of Aviles and McOwen in \cite{aviles1988}.

\begin{lemma}\label{subsollemma}
	Let $(M,g)$ be ALH of regularity order $1$ and with decay exponent $\alpha < n$ and satisfy \eqref{introupperdecay}. For any $0<\beta < \min(n-1, \alpha)$, there exist constants $0<\delta<1$, close to $1$, and $\theta >0$, large, such that the function $u_-\in H^1_{\loc}(M)$ defined by 
	\begin{equation*}
		u_- := \begin{cases}
		1 - C(\theta, \delta)e^{-\alpha r} \qquad &\text{ on } \{r_\delta \leq r \}\times N \\
		1 - \theta e^{-\beta r} \qquad &\text{ on } \{r_\theta\leq r \leq r_\delta\}\times N \\
		0 \qquad \qquad \qquad  \qquad &\text{  on  } M_0\cup (\{r\leq r_\theta\} \times N) \ ,
		\end{cases}
	\end{equation*} where $r_\delta := \frac{1}{\beta}\log \left(\frac{\theta}{1-\delta}\right)$, $r_\theta:= \frac{1}{\beta}\log\left(\theta\right)$ and $C(\theta, \delta) = \left(\frac{\theta}{1-\delta}\right)^{\frac{\alpha}{\beta}}$,
	is a sub-solution of \eqref{yaeqintro} on $M$ which is positive outside of some compact set.
\end{lemma}

We note here that the choice of $C(\theta, \delta)$ ensures that $u_-$ is continuous and therefore belongs to $H^1_{\loc}(M)$. The change at $r_\delta$ is used only in the case that $n-1\leq \alpha < n$.

We have a sub-solution of the form $u_- = u_-(r)$ so we are led to consider the ordinary differential inequality arising immediately from \eqref{yaeqintro}
\begin{align}\label{odeyamabe}
\nonumber-c_n \left(1 + \mathcal{O}(e^{-2\alpha r})\right)u_-^{\prime\prime} &+ c_n\left((n-1)\frac{f_k^\prime}{f_k} + \mathcal{O}(e^{-\alpha r}) \right)u_-^\prime\\
 -& \left(n(n-1) + \mathcal{O}(e^{-\alpha r})\right)u_- \leq -n(n-1)u_-^{\frac{n+2}{n-2}} .
\end{align}

\begin{proof}

	Fix some $0<\beta< \min(n-1, \alpha)$. To establish that $u_-$ is a sub-solution, we must show that \eqref{odeyamabe} holds on $\{r_\theta\leq r\leq r_\delta\}$ and $\{r_\delta \leq r \}$ and check the following transmission conditions at $r_\theta$ and at $r_\delta$,
	\begin{equation*}
		\lim_{r\nearrow r_\theta} u_-^\prime(r) \leq \lim_{r\searrow r_\theta} u_-^\prime(r) \qquad\text{and}\qquad \lim_{r\nearrow r_\delta} u_-^\prime(r) \leq \lim_{r\searrow r_\delta} u_-^\prime(r).
	\end{equation*}
	The first condition is immediate as $\beta > 0$ and the second condition holds true as $\beta < \alpha$.
	
	To establish \eqref{odeyamabe} it suffices to show, for some constant $C_1>0$ depending only on $g$, that there exist $\theta$ and $\delta$ (possibly depending on $\beta$) such that
	\begin{align}\label{convenientineq1}
		\begin{split}
		L_-u_- := &-c_n u_-^{\prime\prime} - c_n\left((n-1)\frac{f_k^\prime}{f_k} - C_1e^{-\alpha r}\right)u_-^\prime\\ &-\left(n(n-1) - C_1e^{-\alpha r}\right)u_- + C_1e^{-2\alpha r}|u_-^{\prime\prime}|
	 	\leq -n(n-1)u_-^{\frac{n+2}{n-2}}
		\end{split}
	\end{align}
	holds on $\{r_\theta< r< r_\delta\}$ and $\{r_\delta < r \}$. We note here that $u_-^\prime \geq 0$ for all $r$.
	
	In the following, we write $C$ to indicate a constant changing from line to line but depending only on $g$. For $\{r_\theta < r < r_\delta\}$,  we have that $n(n-1)u_-^{\frac{n+2}{n-2}} \leq n(n-1)\delta^{\frac{4}{n-2}} u_-$ as $0\leq u_- \leq \delta$. We compute 
	\begin{align*}
	L_-u_- + n(n-1)u_-^{\frac{n+2}{n-2}} < & \ c_n\theta\left(\beta^2 - (n-1)\beta + \frac{1}{4}n(n-2)(1 - \delta^{\frac{4}{n-2}})\right)e^{-\beta r}\\
	& +C(\underbrace{e^{(\beta- \alpha) r} +\theta e^{-\alpha r} + \theta e^{-2r}}_{A})e^{-\beta r}
	\end{align*}
	where we used that $\frac{f_k^\prime(r+r_0)}{f_k(r+r_0)} > 1 - Ce^{-2r}$ and $u_-^\prime \geq 0$. Note first that, for $\delta$ close to $1$, the first term of the RHS of the above is a negative multiple of $e^{-\beta r}$. In addition, as $\beta < \alpha$ we have, for $r\geq r_\theta$, that $|A(r)| \leq A(r_\theta) \leq  C(\theta^{1-\frac{\alpha}{\beta}}+\theta^{1-\frac{2}{\beta}}).$ Consequently, for $\theta$ sufficiently large, the first term on the RHS dominates and, for all $\delta$ sufficiently close to 1, we have $L_-u_- + n(n-1)u_-^{\frac{n+2}{n-2}}<0 \text{ in } \{r_\theta < r < r_\delta\}.$ We choose such a $\theta$ and fix it from here on.
	
	On $\{r > r_\delta\}$, note that we have $\delta \leq u_- \leq 1$. As
	\begin{align*}
	\left|x^{\frac{n+2}{n-2}} - 1 - \frac{n+2}{n-2}(x-1)\right| = &\left|\int_{x}^{1} 4\frac{n+2}{(n-2)^2} t^{-\frac{n-6}{n-2}} (x-t) \ dt\right| < 4\frac{n+2}{(n-2)^2}\delta^{-\frac{n-6}{n-2}}(x-1)^2
	\end{align*}
	for $\delta < x \leq 1$. We have,
	\begin{align*}
	n(n-1)u_-^{\frac{n+2}{n-2}}< n(n-1) - c_n \frac{n}{4}(n+2)C(\theta, \delta) e^{-\alpha r} + c_n\frac{n(n+2)}{(n-2)}\delta^{-\frac{n-6}{n-2}}C(\theta, \delta)^2 e^{-2\alpha r}\ .
	\end{align*} 
	Using again that $\frac{f_k^\prime(r+r_0)}{f_k(r+r_0)} > 1 - Ce^{-2r}$ and $u_-^\prime \geq 0$, we obtain
	\begin{align}\label{rhsbound}
	\nonumber L_-u_-  + n(n-1)u_-^{\frac{n+2}{n-2}}  &\leq  c_n e^{-\alpha r}C(\theta, \delta)\Big[\alpha^2 - (n-1)\alpha -n\\
	& +C\Big(\underbrace{\delta^{-\frac{n-6}{n-2}}C(\theta, \delta) e^{-\alpha r} + e^{-\alpha r}+e^{-2 r}+ \frac{1}{C(\theta, \delta)}}_B\Big)\Big].
	\end{align}
	As $\alpha < n$, $\alpha^2 - (n-1)\alpha -n < 0$. We note that $B(r)$ is non-increasing. We have for $r>r_\delta$.
	\begin{align*}
	0<B(r)<B(r_\delta) = \Bigg[\delta^{-\frac{n-6}{n-2}}(1-\delta)+\left(\frac{\theta}{1-\delta}\right)^{-\frac{\alpha}{\beta}}+\left(\frac{\theta}{1-\delta}\right)^{-\frac{2}{\beta}} + \frac{1}{C(\theta, \delta)}\Bigg] \rightarrow 0 \text{ as } \delta\nearrow 1.
	\end{align*}
	It follows that for $\delta$ close to 1, $L_-u_-  + n(n-1)u_-^{\frac{n+2}{n-2}} <0$ as required.	
\end{proof}

We note that if $\alpha = n$, the leading term in \eqref{rhsbound} vanishes and so the above computation does not work 	. However, it is possible to adapt the proof to include the case that $\alpha = n$ by using a decay rate of $re^{-nr}$ for $r>r_\delta$ and making the corresponding adjustments. A similar obstacle occurs in the later construction of a corresponding super-solution and can be treated with the same adjustment.

We now use the above sub-solution to prove the existence part of Theorem \ref{thmA}.
\begin{prop}\label{part1lemma}
	Let $(M,g)$ be ALH of regularity order $1$ and with decay exponent $\alpha\in(0,n)$ satisfying \eqref{introupperdecay}. There exists a positive smooth solution $u$ of \eqref{yaeqintro} on $M$ satisfying $u \geq 1 - Ce^{-\alpha r}$ for some $C>0$. Consequently, there exists a complete conformal metric $\tilde{g}$ such that $S_{\tilde{g}} \equiv -n(n-1)$ on $M$.
\end{prop}

\begin{proof}
	By Lemma \ref{subsollemma} there exists a sub-solution $u_-$ of \eqref{yaeqintro} which satisfies $u_- \geq 1 - Ce^{-\alpha r}$. The sub- and super-solution argument of Aviles and McOwen in \cite[Proposition 2.1]{aviles1988} yields a smooth solution $u$ of \eqref{yaeqintro} satisfying $u \geq u_-$ on $M$. It remains to show that $u$ is positive on all of $M$; once this is established, we will have obtained a conformal metric $\tilde{g} = u^{\frac{4}{n-2}}g$ which is complete, from the sub-solution lower bound, and has constant scalar curvature.
	
	As $u\geq u_-$, $u$ is non-negative and positive outside of some compact set. Let $B$ be a large ball on which $u\not\equiv 0$ and outside of which $u>0$. On $B$, the scalar curvature $S_g \leq A$ for some constant $A>0$. Consequently, taking a sufficiently large constant $C>0$ such that $u$ satisfies 
	\begin{equation*}
	c_n\Delta_g u - Cu \leq [n(n-1)u^{\frac{n+2}{n-2}} + Au] - Cu \leq 0,
	\end{equation*}
	we can apply the strong maximum principle (\cite[Theorem 3.5]{gilbarg2001elliptic}) to deduce that, as $u\not\equiv 0$ on $B$, $u$ is strictly positive in $B$ and so on all of $M$.
\end{proof}

\subsection{Asymptotic properties of the conformal factor}\label{asymsec}

In order to complete the proof of the remaining parts of Theorem \ref{thmA}, we study the asymptotic properties of solutions of the Yamabe equation as well as the asymptotic properties of the particular conformal factor obtained as the solution to \eqref{yaeqintro} in the previous sub-section.

\subsubsection{An upper bound at infinity}\label{aprioriup}

We will first establish an a priori upper bound on solutions to the Yamabe equation for ALH manifolds in the sense of Definition \ref{alhdef}.
\begin{prop}\label{asymupperprop}
	Let $(M,g)$ be ALH in the sense of Definition \ref{alhdef} and suppose that
	\begin{equation}\label{sclower}
	\liminf_{r(x)\rightarrow \infty} S_g(x) \geq -n(n-1).
	\end{equation}
	Then all solutions of the \eqref{yaeqintro} on $(M,g)$ satisfy $\limsup_{r(x)\rightarrow\infty}u(x) \leq 1$.
\end{prop}

A key point for our discussion will be to compare the Laplacian of the perturbed metric $g$ to that of the reference metric $\mathring{g}$. In particular, where we use $\nabla$ and $\mathring{\nabla}$ to denote the covariant derivative with respect to $g$ and $\mathring{g}$ respectively, we compute directly that
\begin{align}
\Delta_g \varphi &= \mathring{g}^{ij}\mathring{\nabla}_i\mathring{\nabla}_j + (g^{ij} - \mathring{g}^{ij})\mathring{\nabla}_i\mathring{\nabla}_j \varphi + g^{ij}(\Gamma^k_{ij} - \mathring{\Gamma}^k_{ij})\partial_k \varphi \nonumber \\&
= \Delta_{\mathring{g}} \varphi + a^{ij}\mathring{\nabla}_{ij} \varphi + b^i\mathring{\nabla}_i \varphi \label{perturbedlapl}
\end{align}
where $a^{ij} := g^{ij} - \mathring{g}^{ij}$ and $b^i := g^{jk}(\Gamma^i_{jk} - \mathring{\Gamma}^i_{jk})$ satisfy $a^{ij} = \mathcal{O}_1(e^{-(\alpha +2)r})$ and $b^i = \mathcal{O}(e^{-\alpha r})$ using the estimates for the Christoffel symbols found in the proof of Lemma \ref{scalcurvlemma}.

We will make use of the following consequence of the maximum principle:

\begin{lemma}\label{mplemma}
	Let $u>0$ be a bounded smooth solution to \eqref{yaeqintro} on a bounded open set $\Omega\subset M$ and $\bar{u}>0$ be a smooth super--solution to \eqref{yaeqintro} on $\Omega$ satisfying $\bar{u}(x)\rightarrow\infty$ as $x\rightarrow \partial\Omega$. Then necessarily $u<\bar{u}$ on $\Omega$.
\end{lemma}

\begin{proof}
	Due to the facts that $\bar{u}\rightarrow\infty$ as $x\rightarrow\partial\Omega$, $u$ is bounded and both $u>0$ and $\bar{u}>0$, there must exist some $C>0$ such that the difference $w_C := C\bar{u} - u$ satisfies that $w_C\geq 0$ and achieves $0$ at some point in $\Omega$. We claim that $C<1$. If $C\geq1$, then $w_C$ would satisfy,
	\begin{align*}
	-c_n \Delta_g w_C - n(n-1)w_C &\geq -n(n-1)\left(C\bar{u}^\frac{n+2}{n-2} - u^\frac{n+2}{n-2}\right)\\ &\geq-n(n-1)\left((C\bar{u})^\frac{n+2}{n-2} - u^\frac{n+2}{n-2}\right) = c(x) w_C
	\end{align*}
	where
	\begin{equation*}
	c(x) := \begin{cases}
	\frac{(C\bar{u})^\frac{n+2}{n-2} - u^\frac{n+2}{n-2}}{C\bar{u} - u} \text{ for } w_C \neq 0,\\
	\frac{n+2}{n-2}(C\bar{u}(x))^\frac{4}{n-2} \text{ if } w_C = 0.
	\end{cases}
	\end{equation*}
	We could then apply the strong maximum principle to the linear differential inequality $-c_n \Delta_g w_C - \left( n(n-1) + c(x)\right)w_C \geq 0,$ noting that the sign of $\left( n(n-1) + c(x)\right)$ does not matter as the value of the minimum in question is $0$, see \cite[Section 3.2]{gilbarg2001elliptic}. We would then have that $w_C\equiv 0$, a contradiction as $w_C\rightarrow\infty$ at the boundary. We conclude that $C<1$ as claimed and so $u < \bar{u}$ in $\Omega$.
\end{proof}

In order to prove the upper bound in Proposition \ref{asymupperprop}, we will study a family of ODE solutions on annuli in the following series of lemmas.

\begin{lemma}\label{acfu1lemma}
    There exists a positive solution $u_1$ of the equation
	\begin{equation}\label{u1ode}
	-c_n \left(u_1^{\prime\prime} + (n-1)u_1^\prime\right) -n(n-1)u_1 + n(n-1)u_1^\frac{n+2}{n-2} = 0 \text{ on } (-1,1)
	\end{equation}
	satisfying $u_1(r)\rightarrow\infty$ as $r\rightarrow \pm 1$ and there exists a constant $C>0$ depending on $n$ such that $u_1$ satisfies
	\begin{align}\label{u1rates}
		\begin{split}
		\frac{1}{C}(1-|r|)^{-\frac{n-2}{2}}\leq \ & u_{1}(r) \leq C(1-|r|)^{-\frac{n-2}{2}},\\
		&u_{1}^\prime(r) \leq C(1-|r|)^{-\frac{n}{2}},\\
		&u_{1}^{\prime\prime}(r) \leq C(1-|r|)^{-\frac{n+2}{2}}.
		\end{split} 
	\end{align}
\end{lemma}

\begin{proof}
	This result is essentially due to \cite{MR1186044}. Consider the model locally hyperbolic manifold $(\R\times \mathbb{T}^{n-1}, g_\delta)$ where $(\mathbb{T}^{n-1}, \mathring{h})$ is the flat Torus. Define the annulus $A_1:= \{|r|\leq 1\}\subset M$. We note here that any solution of \eqref{u1ode} would automatically be a solution of the Yamabe equation $$-c_n\Delta_{g_\delta} u_1 -n(n-1)u_1 + n(n-1)u_1^{\frac{n+2}{n-2}}=0$$ on the annulus in this particular model space. From \cite[Theorem 1.2]{MR1186044}, we know that, on the annulus $A_1$, there exists a unique solution $u_1$ of the Yamabe equation satisfying $u_1\rightarrow\infty$ at the boundary $\partial A_1$. Additionally, from Theorem 1.3 of \cite{MR1186044}, we know the behaviour of $u_1$ at the boundary satisfies \eqref{u1rates}.
	
	We note that the Yamabe equation is invariant under symmetries of the Torus and so, by uniqueness, $u_1$ must be radial and so the Yamabe equation reduces to the ODE \eqref{u1ode} as desired.
\end{proof}

Equation \eqref{u1ode} has the following important scaling property that we will make use of:
if we define the family $\{u_R\}$ by
\begin{equation}\label{uRdefn}
u_R(r) = \left(\frac{(e^{2} - 1)e^{R}}{e^{2+R}+ e^{2R+r}-e^{2+r}-e^R}\right)^\frac{n-2}{2}u_{1}\left(\log\left(\frac{(e^{2R}-1)e^{r+1}}{e^{2+R}+ e^{2R+r}-e^{2+r}-e^R}\right)\right),
\end{equation}
then each $u_R$ provides a solution of the ODE
\begin{equation}\label{uRode}
-c_n \left(u_R^{\prime\prime} + (n-1)u_R^\prime\right) -n(n-1)u_R + n(n-1)u_R^\frac{n+2}{n-2} = 0 \text{ on } (-R,R).
\end{equation}
Furthermore, from the growth rate of $u_1$ in \eqref{u1rates} we have that each $u_R$ satisfies that $u_R(r)\rightarrow\infty$ as $r\rightarrow \pm R$ and, for some constant $C_R>0$ depending on $R$ and $n$, that
\begin{align}\label{uRrates}
\begin{split}
\frac{1}{C_R}(R-|r|)^{-\frac{n-2}{2}}\leq \ & u_{R}(r) \leq C_R(R-|r|)^{-\frac{n-2}{2}},\\
&u_{R}^\prime(r) \leq C_R(R-|r|)^{-\frac{n}{2}},\\
&u_{R}^{\prime\prime}(r) \leq C_R(R-|r|)^{-\frac{n+2}{2}}.
\end{split} 
\end{align}
We now establish locally uniform convergence of this family to $1$ as $R\rightarrow\infty$.

\begin{lemma}\label{uRlemma}
	The family $\{u_R\}_{R>0}$ of positive solutions of the equation \eqref{uRode} defined above is decreasing with respect to $R$ and satisfies $u_R\searrow1$ uniformly on compact sets as $R\rightarrow\infty$.
\end{lemma}

\begin{proof}
	
	If $R_1 < R_2$, we may apply Lemma \ref{mplemma} on the annulus $A_{R_1}(0)$ to see that $u_{R_2} < u_{R_1}$. Consequently, $\{u_R\}$ is monotone decreasing. 
	
	Define the pointwise limit $u_\infty(r) := \lim_{R\rightarrow\infty} u_R(r)$. We now show that $\{u_R\}$ converges locally in $C^2$ to $u_\infty$. To this end, a locally uniform bound on the first derivative will suffice, as from this we immediately obtain local boundedness of $u_R^{\prime\prime}$ and $u_R^{\prime\prime\prime}$ (the latter after differentiating the ODE once). The Arzela-Ascoli theorem then provides the desired convergence. To obtain such a bound on the first derivative, note that on any compact domain $[-R_0, R_0]$ and for $R > R_0 + \varepsilon$, the fact that $\{u_R\}$ is monotone decreasing implies
	\begin{equation*}
	u_R^{\prime\prime} + (n-1)u_R^\prime = \frac{n(n-2)}{4}\left(u_R^\frac{n+2}{n-2} - u_R\right)
	\end{equation*}
	is bounded in $[-R_0, R_0]$ uniformly as $R\rightarrow\infty$. Integrating the above from some point $s\in[-R_0, R_0]$ to $r\in[-R_0, R_0]$ we obtain that $u_R^\prime(r) - u_R^\prime(s)$ is bounded in $[-R_0, R_0]$ uniformly as $R\rightarrow\infty$. Integrating the above, now in $s$, from $-R_0$ to $R_0$ we see that $u_R^\prime$ is bounded in $[-R_0, R_0]$ uniformly as $R\rightarrow\infty$.
	
	We proceed to show that $u_\infty$ is constant by using the scaling property \eqref{uRdefn} of the family $u_R$. In particular, for any $R,S > 0$ we can verify that
	\begin{equation}\label{scalingproperty}
		u_R(r) = \left(\frac{(e^{2S} - 1)e^{R}}{e^{2S+R}+ e^{2R+r}-e^{2S+r}-e^R}\right)^\frac{n-2}{2}u_{S}\left(\log\left(\frac{(e^{2R}-1)e^{r+S}}{e^{2S+R}+ e^{2R+r}-e^{2S+r}-e^R}\right)\right).
	\end{equation}
	
	Consider the value of $u_R$ at $r=0$ and write $R= S + \Lambda$. We then have
	$$u_{S+\Lambda}(0) = \left(\frac{e^{3S + \Lambda} - e^{S + \Lambda}}{e^{3S+\Lambda}+e^{2(S + \Lambda)} - e^{2S}  - e^{S + \Lambda}}\right)^\frac{n-2}{2}u_{S}\left(\log\left(\frac{e^{3S + 2\Lambda}-e^{S}}{e^{3S+\Lambda} + e^{2(S + \Lambda)}-e^{2S}-e^{S + \Lambda}}\right)\right).$$
	Taking the limit as $S\rightarrow\infty$ while keeping $\Lambda$ fixed we then obtain that for all $\Lambda$, $u_{\infty}(0) = u_{\infty}(\log\left(e^\Lambda\right)) = u_{\infty}(\Lambda)$. Consequently, $u_\infty \equiv C_\infty$ for some constant $C_\infty \geq 0$. 
	
	Recall that $u_R$ converges locally in $C^2$ to $u_\infty$ and so $u_\infty$ solves the Yamabe equation. Consequently, as $u_\infty$ is constant, either $u_\infty \equiv 0$ or $u_\infty \equiv 1$. 
	
	It then suffices to show that $u_\infty (0) > 0$ to conclude the proof. We take the limit in \eqref{uRdefn} as $R\rightarrow\infty$ at $r=0$ to obtain
	$$\lim_{R\rightarrow\infty}u_R(0) = \lim_{R\rightarrow\infty} e^{-\frac{n-2}{2}R}u_{1}\left(1 - (e^{2}-1)e^{-R} + \mathcal{O}(e^{-2R})\right)$$
	and we note also that $1 - (e^{2}-1)e^{-R} + \mathcal{O}(e^{-2R}) \rightarrow 1.$ Then, from the asymptotic rates \eqref{u1rates} for $u_1$, we conclude that
	$$u_\infty(0) =  \lim_{R\rightarrow\infty}u_R(0) \geq \lim_{R\rightarrow\infty} Ce^{-\frac{n-2}{2}R}\left((e^{2}-1)e^{-R} + \mathcal{O}(e^{-2R})\right)^{-\frac{n-2}{2}} = C(e^{2} -1)^{-\frac{n-2}{2}}>0.$$
	Consequently, we conclude that $u_\infty \equiv 1$.
\end{proof}

Having established the properties of the family $u_R$ above, we now prove Proposition \ref{asymupperprop} by showing that these functions can be used to generate super-solutions with the desired properties on annuli in the asymptotic region.

\begin{proof}[Proof of Proposition \ref{asymupperprop}]
	First, we consider the Yamabe equation for $g$ as a perturbation of an elliptic equation with respect to the reference metric $\mathring{g}$ as in \eqref{perturbedlapl}. Fix $R>0$ and consider some point $x_\ast\in M^+$ such that $r_\ast := r(x_\ast) > R$. To bound the value of $u$ at $x_\ast$, we define the annulus $\Omega_R := \Omega(R,r_\ast) = \{x\in M^+ \ : \ |r(x)-r_\ast|<R\}$ and define a candidate super-solution to \eqref{yaeqintro} on $\Omega_R$, $$\bar{u}(x) := Au_R(r(x) - r_\ast)$$ where $u_R$ is defined in \eqref{uRdefn} and $A>0$ is some constant to be determined. In the argument below, all implicit constants in various $\mathcal{O}$ terms are independent of both $A$ and $R$.
	
	We first note that $$-c_n\Delta_{\mathring{g}}\bar{u} = -c_n\left(\partial_{rr} \bar{u} + (n-1)q_k(r)\partial_r\bar{u}\right) = A\left(-c_n\left(u_R^{\prime\prime} + (n-1)u_R^\prime\right) + \mathcal{O}(e^{-2r}|u_R^\prime|)\right)$$
	where
	\begin{equation*}
	q_k(r) = \begin{cases}
	\coth(r) & \ k=1,\\
	1 & \  k=0,\\
	\tanh(r) & \ k=-1,
	\end{cases}
	\end{equation*}
	and $k$ is as in the definition of reference metrics in \eqref{referencemetric} and we use only that $q_k = 1 + \mathcal{O}(e^{-2r})$.
	
	Consequently, from \eqref{uRode}, we have that $\bar{u}$ satisfies
	\begin{align*}
	-c_n\Delta_{\mathring{g}} \bar{u} - n(n-1)\bar{u} + n(n-1)\bar{u}^{\frac{n+2}{n-2}} = (A^{\frac{n+2}{n-2}} - A)u_R^{\frac{n+2}{n-2}} + A\mathcal{O}(e^{-2r}|u_R^\prime|)
	\end{align*}
	for the reference metric $\mathring{g}$ on $\Omega_R \subset M^+$ and $\bar{u}$ blows up as $x\rightarrow \partial \Omega_R$. If follows that
	\begin{align*}
	-&c_n\Delta_{\mathring{g}} \bar{u} +S_g \bar{u} +n(n-1)\bar{u}^{\frac{n+2}{n-2}} + a^{ij}\mathring{\nabla}_{ij}\bar{ u} + b^i\mathring{\nabla}_i \bar{u}\\
	&=n(n-1)(A^{\frac{n+2}{n-2}} - A)u_R^{\frac{n+2}{n-2}} + A\left( a^{ij}\mathring{\nabla}_{ij}u_R + b^i\mathring{\nabla}_i u_R +(S_g+n(n-1))u_R\right) \\
	&\geq n(n-1)(A^{\frac{n+2}{n-2}} - A)u_R^{\frac{n+2}{n-2}} - A \varepsilon(r_\ast)\left(|u_R^{\prime\prime}| + |u_R^\prime| + |u_R|\right)
	\end{align*}
	where $\varepsilon(r) \geq 0$ and $\lim_{r\rightarrow \infty} \varepsilon(r) = 0$; here we have used \eqref{sclower} as well as the asymptotic local hyperbolicity.
	
	From the asymptotic rates in \eqref{uRrates}, we see that there exists a constant $C_R>0$ depending on $R$ such that $|u_R^{\prime\prime}| + |u_R^\prime| + |u_R| \leq C_R u_R^{\frac{n+2}{n-2}}$ near the boundary. Therefore, 
	\begin{align*}
		-c_n\Delta_{\mathring{g}} \bar{u} +S_g \bar{u} +n(n-1)\bar{u}^{\frac{n+2}{n-2}} &+ a^{ij}\mathring{\nabla}_{ij}\bar{ u} + b^i\mathring{\nabla}_i \bar{u} \\&\geq n(n-1)\left(A^{\frac{n+2}{n-2}} - A\left(1+C_R \varepsilon(r_\ast)\right)\right)u_R^{\frac{n+2}{n-2}}.
	\end{align*}
	We can now choose $A = 1 + \tilde{C}\varepsilon(r_\ast)$ with a sufficiently large constant $\tilde{C}$ independent of $x_\ast$ so that the RHS above is positive and so $\bar{u}$ is a strict super-solution to \eqref{yaeqintro}. 
	
	We may now apply Lemma \ref{mplemma} to conclude that $u < \bar{u}$ on $\Omega_R$ and so we have shown that any solution to the Yamabe equation $u$ satisfies $u(x_\ast) < (1+ \tilde{C}\varepsilon(r_\ast))u_R(0).$ Consequently, as $x_\ast$ was arbitrary and $\varepsilon(r_\ast) \rightarrow 0$, we have $\limsup_{r(x)\rightarrow\infty} u(x) \leq u_R(0).$ As $R>0$ was arbitrary, we may now take the limit as $R\rightarrow\infty$ to conclude, from Lemma \ref{uRlemma}, that $\limsup_{r(x)\rightarrow\infty} u(x) \leq 1$ as desired.
\end{proof}

\subsubsection{A lower bound at infinity}\label{apriorilow}

In this sub-section, we provide an a priori lower bound for solutions $u$ of the Yamabe equation for ALH manifolds in the sense of Definition \ref{alhdef}.

We highlight that Example \ref{condneededintro} demonstrates, in a similar way as for uniqueness of the conformal factor, that we must impose some additional condition before attempting to bound the conformal factor $u$ from below. We first provide a lower bound under the assumption that the solution is bounded away from zero at infinity in the sense of \eqref{liminfpos} below. We then refine this first result by weakening the strictly positive requirement of \eqref{liminfpos} to allow some degree of decay of $u$ to zero.

\begin{lemma}\label{liminfprop1}
	Let $(M,g)$ be ALH in the sense of Definition \ref{alhdef} and suppose that
	\begin{equation}\label{sc2}
	\lim_{r(x)\rightarrow \infty} |S_g(x) -n(n-1)| = 0.
	\end{equation}
	Then any solution $u$ of \eqref{yaeqintro} on $(M,g)$ satisfying \begin{equation}\label{liminfpos}
		\liminf_{r(x)\rightarrow\infty} u(x) > 0
	\end{equation} must satisfy $\lim_{r(x)\rightarrow\infty} u(x) = 1.$
\end{lemma}

\begin{proof}
	Suppose, for the sake of contradiction that the conclusion fails. Since, by Proposition \ref{asymupperprop}, we have that $\limsup_{r(x)\rightarrow\infty}u(x) \leq 1$, we must have that $\delta = \liminf_{r(x)\rightarrow\infty} u(x) \in (0, 1)$.
	
	Let $\{(r_i, \theta_i)\}\subset M^+$ be a sequence such that $r_i \rightarrow \infty$ and $u(r_i, \theta_i) \rightarrow \delta$. Passing to a subsequence, we may assume that $\theta_i \rightarrow \theta_\infty$ for some $\theta_\infty\in N$. Take a normal coordinate chart $U$ on $N$ around $\theta_\infty$ with coordinate functions $\theta^a:U\rightarrow\R^n$ so that $\mathring{h}_{ab}(\theta_\infty) = \delta_{ab}$. Consider the region $\Omega :=\R_{\geq0} \times U \subset M^+.$ For each $i$, define a map $\Psi_i: \Omega \rightarrow \R\times\R^{n-1} = \{(\tilde{r}, \tilde{\theta}): \tilde{r}\in\R, \tilde{\theta}\in \R^{n-1}\}$ by $$\Psi_i (r, \theta) = (r-r_i, e^{r_i}(\theta - \theta_i))$$ and define $\tilde{V}_i := \Psi_i(\Omega).$ We also write $\psi_i(\theta) := e^{r_i}(\theta - \theta_i)$. Define $v_i := u \circ \Psi_i^{-1} : \tilde{V}_i \rightarrow \R.$ We note here that $\tilde{V}_1 \subset \tilde{V}_2 \subset \dots$ and $\cup_i \tilde{V}_i = \R \times \R^{n-1}$.
	
	We next compute the equation corresponding to \eqref{yaeqintro} that $v_i$ satisfies. We denote partial derivatives in the coordinates $(\tilde{r}, \tilde{\theta})$ by $\tilde{\partial}$, that is we write $\frac{\partial}{\partial \tilde{r}} = \tilde{\partial_r}$ and $\frac{\partial}{\partial \tilde{\theta}^a} = \tilde{\partial_a}$. On $\R \times \R^{n-1}$ let $g_i = \left(\Psi_i^{-1}\right)^\ast g$. Noting that, by asymptotic local hyperbolicity, we may express $g$ in the form $$g = dr^2 + \varepsilon_{ra}drd\theta^a + \left(f_k(r+r_0)\mathring{h}_{ab} + \varepsilon_{ab}\right)d\theta^ad\theta^b,$$
	where $\varepsilon_{ra}(r, \theta) = \mathcal{O}_m(e^{-(\alpha - 1)r})$ and $\varepsilon_{ab}(r, \theta) = \mathcal{O}_m(e^{-(\alpha - 2)r})$. We thus have
	
	\begin{align*}
	g_i &= d\tilde{r}^2 + e^{-r_i}\varepsilon_{ra}\circ \Psi_i^{-1}d\tilde{r}d\tilde{\theta}^a + e^{-2r_i}\left(f_k^2(\tilde{r}+r_i+r_0)\mathring{h}_{ab}\circ \psi_i^{-1} + \varepsilon_{ab}\circ \Psi_i^{-1}\right)d\tilde{\theta}^a d\tilde{\theta}^b\\
	&= d\tilde{r}^2 + e^{2(\tilde{r} + r_0 )}\mathring{h}_{ab}\circ \psi_i^{-1} d\tilde{\theta}^a d\tilde{\theta}^b + e^{-2r_i}\hat{\varepsilon}_{ab} d\tilde{\theta}^a d\tilde{\theta}^b + e^{-\alpha r_i}\left(\tilde{\varepsilon}_{ra}d\tilde{r}d\tilde{\theta}^a + \tilde{\varepsilon}_{ab}d\tilde{\theta}^a d\tilde{\theta}^b\right)
	\end{align*}
	
	where $\hat{\varepsilon}_{ab} = \hat{\varepsilon}_{ab}^i = \left(f_k^2(\tilde{r} + r_i + r_0) - e^{2(r_i + \tilde{r} + r_0 )}\right)\mathring{h}_{ab}\circ \psi_i^{-1}$, $\tilde{\varepsilon}_{ra} = \tilde{\varepsilon}_{ra}^i = e^{(\alpha-1) r_i}\varepsilon_{ra}\circ \Psi_i^{-1}$ and $\tilde{\varepsilon}_{ab} = \tilde{\varepsilon}_{ab}^i = e^{(\alpha-2) r_i}\varepsilon_{ab}\circ \Psi_i^{-1}$.
	In particular, as $i\rightarrow\infty$, the sequence of matrices of metric components of $g_i$ in $(\tilde{r}, \tilde{\theta})$ is uniformly bounded in $C^1$ and uniformly positive definite on any compact subset of $\R \times \R^{n-1}$. In addition, $\{g_i\}$ converges in $C^1_{\loc}(\R\times\R^{n-1})$ to $g_\infty = d\tilde{r}^2 + e^{2(\tilde{r} + r_0)}\delta_{ab}d\tilde{\theta}^a d\tilde{\theta}^b$.
	
	Each $v_i$ solves the equation
	\begin{equation}\label{vieq}
	-c_n\Delta_{g_i} v_i = s_i v_i-n(n-1)\left(v_i^{\frac{n+2}{n-2}} - v_i\right) = c_i v_i
	\end{equation}
	where $s_i = S_{g_i}\circ \Psi_i^{-1} + n(n-1)$ which, by \eqref{sc2}, satisfies $s_i \rightarrow 0$ in $C^0_{\loc}$ as $i\rightarrow\infty$ and $c_i  = s_i - n(n-1)\left(\left(u\circ\Psi_i^{-1}\right)^{\frac{4}{n-2}} - 1\right)$. By Proposition \ref{asymupperprop}, $u$ is bounded and so $\{c_i\}$ is uniformly bounded on any compact subset of $\R \times \R^{n-1}$. Consequently, fixing some $R>0$, we may apply the Harnack inequality of \cite[Theorem 8.20]{gilbarg2001elliptic} to the equation $L_i v_i = 0$ where $L_i := c_n\Delta_{g_i} + c_i$ to conclude that, on the ball $B_R(0, 0) \subset \R\times\R^{n-1}$,
	\begin{equation}\label{harnack}
		\sup_{B_R(0, 0)} v_i \leq  C_R \inf_{B_R(0, 0)} v_i \leq C_R\delta
	\end{equation}
	where the final inequality follows from the fact that $v_i(0,0) \rightarrow \delta$ by construction and $C$ is some constant independent of $i$ (depending only on $n$, $R$ and bounds on the ellipticity of $L_i$ and sup-norm of the coefficients of $L_i$ which we have shown may be chosen independently of $i$). We conclude that $\{v_i\}$ is uniformly bounded on $B_R(0, 0)$.
	
	After recasting \eqref{vieq} to non-divergence form and noting that $g$ is ALH of regularity order $1$, we may now apply standard $W^{2, p}$ estimates (\cite[Theorem 9.11]{gilbarg2001elliptic}) to conclude that $v_i$ are locally uniformly bounded in $W^{2, p}$ for all $p\in(1, \infty)$ and hence in $C^{1,\beta}$ for all $\beta \in (0, 1)$. Therefore, passing to a subsequence if necessary, we may assume that $v_i \rightarrow v_\infty \text{ in } C^1_{\loc}(\R\times\R^{n-1}).$
	Furthermore, for any point $(\tilde{r}_0, \tilde{\theta}_0)\in \R\times\R^{n-1}$ and corresponding sequence $(\bar{r}_i, \bar{\theta}_i) := \Psi_i^{-1}(\tilde{r}_0, \tilde{\theta}_0)\in M^+$, we have that 
	\begin{align*}
	v_\infty(\tilde{r}_0, \tilde{\theta}_0) &= \lim_{i\rightarrow\infty} (v_i\circ\Psi_i)(\bar{r}_i, \bar{\theta}_i) = \lim_{i\rightarrow\infty} u(\bar{r}_i, \bar{\theta}_i) \geq \delta.
	\end{align*}
	
	We now pass to the limit as $i\rightarrow\infty$ in \eqref{vieq} to obtain 
	\begin{equation}\label{limeqn}
		-c_n \Delta_{g_\infty} v_\infty = -n(n-1)\left(v_\infty^{\frac{n+2}{n-2}} - v_\infty \right) \text{ on }\R\times\R^{n-1},
	\end{equation} and so $v_\infty$ is in fact in $C^\infty(\R\times\R^{n-1})$ (as $v_\infty>0$). By construction, we have that $v_\infty(0,0) = \lim_{i \rightarrow\infty} v_i(\Psi(r_i, \theta_i))  = \lim_{i \rightarrow\infty} u(r_i, \theta_i) = \delta$. Since $\delta$ is a strict sub-solution of \eqref{limeqn} and $v_\infty \geq \delta$ on $\R\times\R^{n-1}$, we obtain a contradiction to the strong comparison principle.
\end{proof}

We are now able to prove our main result of the section, expanding the extent of the lower bound result by allowing for the conformal factor to exhibit decay to zero at infinity.

\begin{prop}\label{decayliminfprop}
	Let $(M,g)$ be an ALH manifold in the sense of Definition \ref{alhdef} which satisfies \eqref{sc2}. Suppose that $u$ is a solution of the Yamabe problem on $(M,g)$ and that there exist an $S > 0$ and $r_1 > r_0$ such that
	\begin{equation}\label{odeassump}
		\frac{\underbar{u}(r+S)}{\underbar{u}(r)} > \frac{n}{2}e^{-\frac{n-2}{2}S} - \frac{n-2}{2} e^{-\frac{n}{2}S} \text{ for all }r > r_1,
	\end{equation}
	where $\underline{u}(r) := \min_N u(r, \cdot)$. Then $\lim_{r(x)\rightarrow\infty} u(x) = 1.$
\end{prop}

We note that the asymptotic exponential rate of $-\frac{n-2}{2}$ is also seen in Example \ref{condneededintro} and represents a borderline rate where a solution to \eqref{yaeqintro} exists with such a rate of decay but notably the corresponding conformal metric fails to be complete. Indeed, for any slower decay, we obtain the following immediate corollary of Proposition \ref{decayliminfprop}:

\begin{corollary}\label{nonexistcorol}
	Let $(M,g)$ be an ALH manifold in the sense of Definition \ref{alhdef} which satisfies \eqref{sc2}. There does not exist a solution $u$ of \eqref{yaeqintro} on $M$ satisfying $e^{\beta r(x)} u(x) \rightarrow 0$ as $r(x) \rightarrow \infty$ for any $0 < \beta < \frac{n-2}{2}$.
\end{corollary}

\begin{proof}[Proof of Proposition \ref{decayliminfprop}]
	By Lemma \ref{liminfprop1}, it suffices to show that $\liminf_{r(x)\rightarrow\infty} u(x) > 0$. Suppose not, then there exists a sequence $r_i \rightarrow \infty$ such that $\underline{u}(r_i)\rightarrow 0$. It follows that $$\min_{r(x)\in[r_1, R]} u(x) \rightarrow 0 \text{ as } R\rightarrow\infty.$$ Consequently, for $R$ sufficiently large the minimum cannot be attained at $r_1$. Furthermore, the fact that $u\leq 1$ (by Proposition \ref{asymupperprop}) and $\liminf u = 0$ together imply, via the strong maximum principle, that $u < 1$ and satisfies $$-c_n \Delta_g u = -n(n-1)\left(u^\frac{n+2}{n-2} - u\right) > 0.$$ The strong maximum principle again implies that the minimum of $u$ on $\{r(x)\in [r_1, R]\}$ is achieved only on $\{r(x) = R\}$ and so $\underline{u}$ is strictly decreasing for large $r$.
	
	Define $\alpha_i := \underline{u}(r_i)^{-1} \rightarrow \infty$ and consider $w_i := \alpha_i v_i$ where $v_i$ is as in the proof of Lemma \ref{liminfprop1}. Note that, by \eqref{vieq}, $w_i$ satisfies $-c_n\Delta_{g_i} w_i = c_i w_i.$ As in the proof of Lemma \ref{liminfprop1}, we have that $w_i \rightarrow w_\infty$ in $C^1_{\loc}(\R\times\R^{n-1})$.
	Since $\alpha_i \rightarrow \infty$, the limit $w_\infty$ now solves
	\begin{equation}\label{infeqw}
	-c_n\left(\tilde{\partial}_{rr} w_\infty + (n-1)\tilde{\partial}_r w_\infty + e^{-2(\tilde{r} + r_0)}\Delta_{n-1}w_\infty\right) = n(n-1)w_\infty \text{ on } \R\times\R^{n-1}.
	\end{equation}
	Moreover, by construction, $w_\infty(r, \theta) \geq 1$ for $r\leq 0$ and there exists a $\theta_0$ such that $w_\infty(0, \theta_0) = 1$. We will now see that the above provides a contradiction to \eqref{odeassump}. Indeed, consider the radial function $$\underline{w}(r) = \frac{n}{2}e^{-\frac{n-2}{2}r} - \frac{n-2}{2} e^{-\frac{n}{2}r}$$ which solves \eqref{infeqw} and satisfies $\underline{w}(0) = 1$ and $\underline{w}'(0) = 0$. By assumption \eqref{odeassump}, there exists an $S > 0$ such that $w_\infty(S, \theta) > \underline{w}(S)$ for all $\theta\in N$. Consequently, as $w_\infty(0, \theta) \geq 1 = \underline{w}(0)$, the strong maximum principle and the Hopf Lemma imply that we must have $w_\infty > \underline{w}$ in the region $\{r(x) \in (0, S)\}$ and $\partial_r w_\infty(0, \theta_0) > 0$. The latter contradicts the fact that $w_\infty (r, \theta) \geq 1$ for $r \leq 0$ and $w_\infty(0, \theta_0) = 1$.
\end{proof}

\subsubsection{First derivative estimates at infinity}

In this sub-section, we prove results for the first derivative decay of the conformal factor. 

We begin by defining a particular super-solution on $M$. The following super-solution will provide an asymptotic upper bound on our solution obtained in Proposition \ref{part1lemma} and will allow us to conclude that the conformal metric remains locally asymptotically hyperbolic to the same order.
\begin{lemma}\label{supersollemma}
	Let $(M,g)$ be ALH of regularity order $1$ and with decay exponent $\alpha \in (0, n)$ and $S_g \geq -n(n-1) - Ce^{-\alpha r}$ for some constant $C>0$. There exist constants $A>0$ and $R>0$ such that the function $u_+\in H^1_{\loc}(M)$ defined by
	\begin{equation*}
		u_+(r) := \begin{cases}
		1+Ae^{-\alpha r} \qquad \qquad &\text{ on } \{r\geq R\} \times N\\
		1+Ae^{-\alpha R} \qquad\qquad &\text{ on } M_0\cup (\{r < R\} \times N)
		\end{cases} ,
	\end{equation*}
	is a super-solution to \eqref{yaeqintro} on $M$.
\end{lemma}

\begin{proof}
 	We first note that the transmission condition $ 0 = \lim_{r\nearrow R} u_+^\prime(r) \geq \lim_{r\searrow R} u_+^\prime(r) =  -\alpha A e^{-\alpha R}$ certainly holds. Thus, we only need to find large $A$ and $R$ such that $$-c_n\Delta_g u_+ + S_g u_+ \geq -n(n-1)u_+^\frac{n+2}{n-2}$$ holds in the two regions $r\leq R$ and $r > R$.
	
	We proceed similarly as in the proof of Lemma \ref{subsollemma} with the difference that, as $u_+>1$, we can estimate directly on the RHS that
	\begin{equation*}
	-n(n-1)u_+^{\frac{n+2}{n-2}} \leq -n(n-1)\left(1 + \frac{n+2}{n-2}(u_+ - 1)\right).
	\end{equation*}
	In the following the value of $C$ may change from line to line but always depends only on $g$ and we choose $R$ larger if necessary so that $R>1$. We compute for $r>R$, using that $\frac{f_k^\prime}{f_k}(r+r_0) > 1 - Ce^{-2r}$ and $u_+^\prime < 0$,
	\begin{align*}
	-c_n\Delta_g u_+ + S_g u_+ +& n(n-1)u_+^\frac{n+2}{n-2} \geq \\ &-c_nAe^{-\alpha  r}\Bigg[ \underbrace{\alpha ^2 - (n-1)\alpha - n + Ce^{-\alpha r} + Ce^{-2r} +\frac{C}{A}}_{B}\Bigg].
	\end{align*}
	Thus, we must choose $A$ and $R$ such that $B\leq 0$. As $\alpha ^2 - (n-1)\alpha  - n < 0$, we can select an $A_0$ and $R$ such that the RHS of the above is positive for all $A>A_0$. Fix $R$ from hereon. It remains to find an $A>A_0$ such that \eqref{odeyamabe} holds on $r\leq R$ which can be verified by direct computation and noting that the scalar curvature must be bounded on this compact interior region.
\end{proof}

Having established the super- solution above, we use it to gain control on the solution $u$ obtained in the Section \ref{existsec}.

\begin{lemma}\label{decaylemma}
	Let $(M,g)$ be an ALH manifold of regularity order $1$ and with decay exponent $\alpha \in (0,n)$ satisfying \eqref{introfulldecay}. Then the smooth solution $u$ of \eqref{yaeqintro} on $M$ obtained in Proposition \ref{part1lemma} satisfies $|u -1| + |\mathring{\nabla} u |_{\mathring{g}} = \mathcal{O}(e^{-\alpha r})$. Furthermore, $u$ is maximal in the sense that any solution $\tilde{u}$ of \eqref{yaeqintro} satisfies $\tilde{u} \leq u$.
\end{lemma}

\begin{proof}
	We first prove the sup-norm decay. From Proposition \ref{part1lemma}, we have that $u\geq u_- \geq 1 - Ce^{-\alpha r}$ where $u_-$ is the sub-solution constructed in Lemma \ref{subsollemma}. Additionally, Proposition \ref{asymupperprop} gives that $\limsup_{r(x)\rightarrow \infty} u \leq 1$. Consequently, we have that $\lim_{r(x)\rightarrow \infty} u(x) = 1$
	
	To obtain the rate of decay from above, we show that $u\leq u_+$ where $u_+$ is the super-solution $u_+\leq 1 + Ce^{-\alpha r}$ constructed in Lemma \ref{supersollemma}. For the sake of contradiction, we suppose that $\inf \frac{u_+}{u} = \frac{1}{C} < 1$ for some constant $C>1$. As $\lim_{r(x)\rightarrow \infty} \frac{u_+(x)}{u(x)} = 1$, the infimum would be achieved and so we could define $v:= Cu_+ - u$ which would satisfy $v\geq 0$ and would achieve $0$ at some minimum. We could then apply the same maximum principle argument as in Lemma \ref{mplemma} to conclude that $\inf \frac{u_+}{u} \geq 1$ and so $u \leq u_+$ as desired.

	We next prove the derivative estimates. Let $(r_\ast, \theta_\ast)\in M^+$ be some arbitrary point with $r_\ast$ sufficiently large. Take a normal coordinate chart $U$ on $N$ around $\theta_\ast$ with coordinate functions $\theta^a:U\rightarrow\R^n$ so that $\mathring{h}_{ab}(\theta_\infty) = \delta_{ab}$. Consider the region $$\Omega :=(r_\ast - 1, r_\ast + 1) \times U \subset M^+.$$ Define a map $\Psi: \Omega \rightarrow \R\times\R^{n-1} = \{(\tilde{r}, \tilde{\theta}): \tilde{r}\in\R, \tilde{\theta}\in \R^{n-1}\}$ by $$\Psi (r, \theta) = (r-r_\ast, e^{r_\ast}(\theta - \theta_\ast))$$ and define $\tilde{V} := \Psi(\Omega).$ We also write $\psi(\theta) := e^{r_\ast}(\theta - \theta_\ast)$. Define $v := u \circ \Psi^{-1} : \tilde{V} \rightarrow \R.$
	
	Let $g_\ast = (\Psi^{-1})^* g$. As in the proof of Lemma \ref{liminfprop1}, as $r_\ast \rightarrow \infty$, $g_\ast$ converges in $C^1$ to $g_\infty = d\tilde{r}^2 + e^{2(\tilde{r} + r_0)}\delta_{ab}\ d\tilde{\theta}^a d\tilde{\theta}^b$. Moreover (see equation \eqref{vieq}), $v$ satisfies $\Delta_{g_\ast} v = \mathcal{O}\left(e^{-\alpha  r_\ast}\right)$ where the implicit constant in the $\mathcal{O}$ term is independent of $r_\ast$. Applying $W^{2,p}$ estimates to the above equation, after recasting to non-divergence form, and using Sobolev embeddings, we obtain $|\nabla v(0,0)|_{g_\infty} = \mathcal{O}(e^{-\alpha r_\ast})$ and so obtain decay in the first derivative $|\nabla u|_{\mathring{g}} = \mathcal{O}(e^{-\alpha r})$. 
	
	We finally address the maximality of $u$. Given another solution $\tilde{u}>0$ of \eqref{yaeqintro}, we know that $\limsup_{r(x)  \rightarrow \infty} \tilde{u}(x) \leq 1$, again from Proposition \ref{asymupperprop}, and so we have that $\liminf_{r(x)\rightarrow \infty} \frac{u}{\tilde{u}} \geq 1$ as $\lim_{r(x)\rightarrow \infty} u(x) = 1$. Consequently, if $\inf \frac{u}{\tilde{u}} < 1$ then the infimum must be attained and so we may apply the maximum principle (as in the proof of $u\leq u_+$) to conclude that $\tilde{u} \leq u$.
\end{proof}

\subsection{Uniqueness and completion of the proof of Theorem \ref{thmA}}\label{proofsec}

Under assumption \eqref{introupperdecay}, the first part of Theorem \ref{thmA} is proved in Proposition \ref{part1lemma}. For the remainder of the proof, we will assume that \eqref{introfulldecay} holds. By Lemma \ref{decaylemma}, we know that $\lim_{r\rightarrow\infty} u = 1$, $|u -1| + |\mathring{\nabla} u |_{\mathring{g}} = \mathcal{O}(e^{-\alpha r})$ and $u$ is the maximal solution of \eqref{yaeqintro}. It remains to prove the uniqueness of $u$ and that the conformal metric $\tilde{g} = u^\frac{4}{n-2}g$ is ALH of regularity order $0$ and with decay exponent $\alpha$, which we now treat in turn.

The uniqueness of $u$ is a direct corollary of the stronger result below:

\begin{prop}\label{uniqprop}
	Let $(M,g)$ be ALH of regularity order $1$ and decay exponent $\alpha\in(0, n)$ satisfying \eqref{introfulldecay}. Then the solution $u$ of the Yamabe problem on $(M,g)$ obtained in Proposition \ref{part1lemma} is the unique solution such that there exist an $S > 0$ and $r_1 > r_0$ such that
	\begin{equation}\label{odeassumpunic}
	\frac{\underbar{u}(r+S)}{\underbar{u}(r)} > \frac{n}{2}e^{-\frac{n-2}{2}S} - \frac{n-2}{2} e^{-\frac{n}{2}S} \text{ for all }r > r_1,
	\end{equation}
	where $\underline{u}(r) := \min_N \tilde{u}(r, \cdot)$.
\end{prop}

\begin{proof}
	Let $\tilde{u}$ be any solution of \eqref{yaeqintro} on $(M,g)$ satisfying \eqref{odeassumpunic}. By Proposition \ref{decayliminfprop}, we know that $\lim_{r\rightarrow\infty} \tilde{u} = 1$. Observe that in the proof of the maximality of $u$ in Lemma \ref{decaylemma} we actually proved that any solution of \eqref{yaeqintro} which tends to $1$ at infinity is maximal. Consequently, both $u$ and $\tilde{u}$ are maximal and hence equal.
\end{proof}

Lastly, we prove that the asymptotic local hyperbolicity of the obtained complete conformal metric of constant scalar curvature, thus completing the proof of Theorem \ref{thmA}.

\begin{lemma}\label{asymhyplemma}
	Let $(M,g)$ be an ALH manifold of regularity order $1$ and with decay exponent $\alpha \in (0,n)$ satisfying \eqref{introfulldecay}. Let $u$ be the solution of \eqref{yaeqintro} obtained in Proposition \ref{part1lemma}, then $u = 1 +\mathcal{O}(e^{-\alpha r})$ and the corresponding conformal metric is ALH of regularity order $0$ and with decay exponent $\alpha$.
\end{lemma}

\begin{proof}
	By Lemma \ref{decaylemma} we have $ |u-1| + |\mathring{\nabla} u|_{\mathring{g}}= \mathcal{O}(e^{-\alpha r})$. Consider the new coordinate function
	\begin{equation*}
	z := r + \int_{r}^{\infty} \left(1 - u^\frac{2}{n-2}\right) \ ds \ .
	\end{equation*}
	First note that
	\begin{equation*}
	z - r = \int_{r}^{\infty} \left(1 - u^\frac{2}{n-2}\right) \ ds = \mathcal{O}_1(e^{-\alpha r})
	\end{equation*}
	and the map $(r, \theta) \mapsto (z, \theta)$ is a diffeomorphism for large $r$. We will show that this implies $\tilde{g} = u^{\frac{4}{n-2}}g$ is ALH of regularity order $0$ and decay exponent $\alpha$ with respect to the coordinate system $(z, \theta)$.
	
	We compute that $$dz = u^\frac{2}{n-2} dr - \left(\int_r^\infty \frac{4-n}{n-2}u^{\frac{4-n}{n-2}}\partial_a u \ ds\right)d\theta^a = u^\frac{2}{n-2} dr + \mathcal{O}(e^{-(\alpha -1 ) r}) d\theta^a.$$
	Recalling that $g$ is ALH and so may be written $g = dr^2 + f_k(r+r_0)^2 \mathring{h} + \varepsilon_{ra}d\theta^a dr + \varepsilon_{ab}d\theta^a d\theta^b$ where $\varepsilon_{ab} = \mathcal{O}(e^{-(\alpha-2)r})$ and $\varepsilon_{ra} = \mathcal{O}(e^{-(\alpha -1)r})$, we deduce that
	\begin{equation*}
	\tilde{g} = dz^2 + f_k(z+r_0)^2\mathring{h} + \tilde{\varepsilon}_{za}d\theta^a dz + \tilde{\varepsilon}_{ab}d\theta^ad\theta^b
	\end{equation*}
	where $\tilde{\varepsilon}_{za} = \mathcal{O}(e^{-(\alpha -1)z})$ and
	\begin{equation*}
	\tilde{\varepsilon}_{ab} = u^{\frac{4}{n-2}}\varepsilon_{ab} + \left(u^{\frac{4}{n-2}}f_k(r+r_0)^2 - f_k(z+r_0)^2\right)\mathring{h}_{ab} + \mathcal{O}(e^{-2(\alpha - 1)z}).
	\end{equation*} 
	Since $f_k(r+r_0)^2 = f_k(z+r_0)^2 + \mathcal{O}_1(e^{-(\alpha-2)z})$, the conclusion follows.
\end{proof}

\subsection{Asymptotically warped product manifolds}\label{wpsubsec}

To conclude this section, we briefly outline some remarks on the Yamabe Problem for asymptotically warped product manifolds.

Recall that, in the definition of ALH manifolds, we supposed that we could decompose $(M, g)$ in two parts, $M = M_0 \cup M^+$, where $M_0$ is some compact manifold with boundary, $M^+ = \R^+ \times N$ with $(N,\mathring{h})$ some compact manifold of constant scalar curvature $S_{\mathring{h}} = (n-1)(n-2)k$ for some $k \in \{-1, 0, 1\}$ and $M_0$ and $M^+$ coincide on their common boundary. We then considered those metrics $g$ which asymptote to the locally hyperbolic reference metrics in \eqref{referencemetric}.

One may, instead, consider a different choice of reference metric to which $g$ is asymptotic. An immediate generalisation of interest would be the following warped product metrics
\begin{equation}\label{warped}
\mathring{g}_f = dz^2 + f^2(z) \mathring{h}
\end{equation} 
where $f$ is some positive function, usually referred to as the warping function. Metrics $g$ which are asymptotic to $\mathring{g}_f$ will be informally referred to as asymptotically warped product (AWP) metrics.

In this section, we discuss the applications of Proposition \ref{thmA} in this broader category of AWP metrics. This amounts to understanding whether the warped product metrics of \eqref{warped} are conformal to the reference locally hyperbolic metrics of \eqref{referencemetric}. We establish a necessary and sufficient condition on the warping function for such conformality in Proposition \ref{wpeconfprop}. The condition we identify is an integral condition of Keller-Osserman type (see \cite{https://doi.org/10.1002/cpa.3160100402, pjm/1103043236}). Furthermore, we also show that, if the warped product metric is not conformally locally hyperbolic, then it is conformal to a complete metric of finite volume.

\begin{prop}\label{wpeconfprop}
	Let $f: [0, \infty) \rightarrow \R$ be a smooth, positive function. The metric $\mathring{g}_f$ in \eqref{warped} is conformal to a locally hyperbolic metric \eqref{referencemetric} if and only if
	\begin{equation}\label{wpcond}
	I(f) := \int_0^\infty \frac{1}{f(s)} \ ds < \infty.
	\end{equation}
	Moreover, if \eqref{wpcond} does not hold, then $\mathring{g}_f$ is conformal to $d\check{z}^2 + e^{-2\check{z}} \mathring{h}$ which is a complete metric of finite volume on $M^+$.
\end{prop}

\begin{proof}
	\underline{Step 1:} We first suppose that $I(f) < \infty$ and show that $\mathring{g}_f$ is conformal to a locally hyperbolic metric.
	
	Recall the function $f_k$ in the definition of the locally hyperbolic reference metrics $\mathring{g}$ in \eqref{referencemetric} and consider the following separable ODE for an unknown function $K$: 
	\begin{equation}\label{fKode}
	f(z)K^\prime(z) = f_k(K(z)) =  \begin{cases}
	\sinh(K(z) + r_0) & \ k=1,\\
	e^{K(z)+r_0} & \ k=0,\\
	\cosh(K(z) + r_0) & \ k=-1,
	\end{cases}
	\end{equation}
	with the initial condition $K(z_0) = 0$ where $z_0 = 0$ and $r_0 = 2 \arctanh\left(\exp(-I(f))\right)$ if $k=1$, $z_0=0$ and $r_0 = \log\left(I(f)\right)$ if $k=0$, and $z_0$ satisfies $\int_{z_0}^\infty f(s) \ ds < \frac{\pi}{2}$ and $r_0 = 2\arctanh\left(\tan\left(\frac{\pi}{4} - \frac{1}{2}\int_{z_0}^\infty f(s) \ ds\right)\right)$ if $k=-1$. It is straightforward to check that the explicit solution is defined for all $z\geq z_0$ and, furthermore, that $K(z)$ is strictly increasing and $K(\infty) = \infty$. Consequently, $K:[z_0, \infty)\rightarrow [0, \infty)$ is a diffeomorphism.
	
	Consider a coordinate system $(r, \theta)$ on $M^+$ where $\theta$ is some coordinate system on $N$ and $r = K(z)$ for $z>z_0$. For $z>z_0$, in view of \eqref{referencemetric} and \eqref{fKode}, we have $$\mathring{g}_f = (K^\prime(z))^{-2}\left(dr^2 + (K^\prime(z)f(z))^{2}h\right) = (K^\prime(z))^{-2}\mathring{g}$$
	as desired.
	
	\medskip
	\noindent \underline{Step 2:} In order to show the sufficiency in the proposition, we first prove that $I(f) = \infty$ if and only if $\mathring{g}_f$ is conformal to a complete, finite volume metric on $M^+$.
	
	To see this, first suppose $I(f) = \infty$ and define $\check{K}(z) = \log\left(1 + \int_0^z (f(s))^{-1} ds\right)$ so that $\check{K}^\prime(z) = \left(f(z)\right)^{-1}e^{-\check{K}(z)}$. As $I(f) = \infty$, we may define a new coordinate $\check{z} = \check{K}(z)$ so that $$\mathring{g}_f = dz^2 + f^2(z)\mathring{h} = \left(\check{K}^\prime(z)\right)^{-2}\left(d\check{z}^2 + (\check{K}^\prime(z)f(z))^2 \mathring{h}\right) = \left(\check{K}^\prime(z)\right)^{-2}\left(d\check{z}^2 + e^{-2\check{z}} \mathring{h}\right)$$ and so $\mathring{g}_f$ is conformal to $d\check{z}^2 + e^{-2\check{z}}\mathring{h}$ which can readily be seen to be both complete and of finite volume on $M^+$.
	
	Conversely, suppose that there exists some conformal factor $u(z,\theta)$ such that $\check{g} = u^{\frac{4}{n-2}}(z, \theta) \mathring{g}_f$
	for some complete, finite volume metric $\check{g}$. Consider the divergent curves $\gamma_\theta : [0,\infty) \rightarrow M^+$ defined by $\gamma_\theta (t) = \left(z(t), \theta\right)$ for $\theta\in N$. By completeness, we have that the length $\mathcal{L}_{\check{g}}(\gamma_\theta)=\infty$ for all $\theta \in N$. We compute directly the length
	\begin{align*}
	\infty = \mathcal{L}_{\check{g}}(\gamma_\theta) &= \int_0^\infty \sqrt{\check{g}\left(\dot{\gamma}_\theta(t), \dot{\gamma}_\theta(t)\right)} \ dt = \int_0^\infty u^{\frac{2}{n-2}}(z, \theta) \ dz.
	\end{align*}
	Consequently, we may write
	\begin{align*}
	\infty = \int_N \mathcal{L}_{\check{g}}(\gamma_\theta) \ d\theta &= \int_N \int_0^\infty u^{\frac{2}{n-2}}(z, \theta) \ dz \ d\theta\\
	&\leq \left(\int_N \int_0^\infty u^{\frac{2n}{n-2}}(z, \theta) f^{n-1}(z) \ dz \ d\theta\right)^\frac{1}{n}\left(\int_N \int_0^\infty \frac{1}{f(z)} \ dz \ d\theta\right)^\frac{n-1}{n}\\
	&= Vol_{\check{g}}(M^+)^\frac{1}{n} Vol_h(N)^\frac{n-1}{n}I(f)^\frac{n-1}{n}.
	\end{align*}
	Since $Vol_{\check{g}}(M^+) < \infty$ we deduce that $I(f) = \infty$.
	
	\medskip
	\noindent\underline{Step 3:} We now suppose that $I(f) = \infty$ and prove that then $\mathring{g}_f$ cannot be conformal to a locally hyperbolic metric.
	
	By Step 2, $\mathring{g}_{f}$ (which has $I(f) = \infty$) cannot be conformal to another warped product metric $\mathring{g}_{\tilde{f}}$ with $I(\tilde{f}) < \infty$, as the former is conformal to a complete metric of finite volume, and the latter cannot be. On the other hand, we can compute that $I(f_k) < \infty$ for the reference locally hyperbolic metrics in \eqref{referencemetric} for each $k= -1, 0, 1$. The conclusion follows.
\end{proof}

It should be clear that, via Proposition \ref{wpeconfprop}, one can apply Theorem \ref{thmA} to solve the Yamabe problem for AWP metrics with a wide range of warping functions. For further detail, the reader is referred to \cite{thesis}. The only reference warped product metrics for which this line of argument does not apply are those which are conformal to $\check{g} = d\check{z}^2 + e^{-2\check{z}} \mathring{h}$. It can be shown that $\check{g}$ admits a conformal compactification where infinity corresponds to a single point, which is very different to the situation for ALH metrics.

\section{Volume ratio conditions for solvability of the Yamabe problem}\label{dcsec}

	In this section of this paper, we prove Theorem \ref{dcexistthm} concerning the negativity of the first eigenvalue of the conformal Laplacian on sub-domains of non-compact manifolds of negative curvature type satisfying the volume ratio condition \eqref{dccondexist}. In view of the work of \cite{aviles1988}, this leads us to an existence result for the Yamabe problem. We furthermore show that condition \eqref{dccondexist} is sharp for the negativity of the first eigenvalue (see Proposition \ref{sharpprop}).

    \subsection{An upper bound on the eigenvalue for the conformal Laplacian}
    
    As in the setting of Theorem \ref{dcexistthm}, let $(M,g)$ be a Riemannian manifold and suppose that there exist two open sets $\Omega_1\subset \Omega_2$ with $C^1$ boundary which satisfy \eqref{eqspace} and that the scalar curvature satisfies $S_g \leq -n(n-1)$ on $\Omega_2$.
    
    In proving Theorem \ref{dcexistthm}, we relate the first eigenvalue of the conformal Laplacian on $\Omega_2$ to the following sup-norm minimisation problem: 
    \begin{equation}\label{snmin}
    H_0 = \inf \left\{H(\tilde{\varphi}) := \sup_{0\leq r \leq R} F\left(\tilde{\varphi}(r), \tilde{\varphi}^\prime(r)\right) : \tilde{\varphi}\in C^1([0, R]), \tilde{\varphi}(0) = 0, \tilde{\varphi}(R) = 1\right\}
    \end{equation}
    where $F(y,z) := c_n z^2 - n(n-1) y^2$ and $R$ is given in \eqref{eqspace}. Note that $H(\tilde{\varphi}) \geq F(\tilde{\varphi}(0), \tilde{\varphi}'(0)) \geq 0$ and so $H_0 \geq 0$.
    
    We recall the variational formulation of the first eigenvalue $\lambda$ of the conformal Laplacian with Dirichlet boundary data on the bounded domain $\Omega_2$:
    \begin{equation}\label{varform}
    \lambda = \inf_{\substack{\varphi\in H^1_0(\Omega_2), \\ \|\varphi\|_{L^2} = 1}} \int_{\Omega_2} \left(c_n|\nabla_g \varphi|^2 + S_g\varphi^2\right) dV_g.
    \end{equation}
    
    Define the distance function $r:\overline{\Omega_2\setminus\Omega_1} \rightarrow [0,R]$ by $r(x) = d_g(x, \partial\Omega_2).$ Central in our discussion will be test functions $\varphi = \varphi(r)$ of the form
    \begin{equation}\label{phiform}
    \varphi(x) = \left\{\begin{array}{ll}\displaystyle
    1 & x\in\Omega_1\\
    \vspace{-0.3cm}\\
    \tilde{\varphi}(r(x)) & x\in\overline{\Omega_2\setminus\Omega_1}
    \end{array}\right.
    \end{equation}
    where $\tilde{\varphi}:[0, R] \rightarrow \R$ is a $C^1$ function satisfying $\tilde{\varphi}(0) = 0$ and $\tilde{\varphi}(R) = 1$. We note that, as $r$ is Lipschitz, these conditions on $\tilde{\varphi}$ ensure that $\varphi\in H^1_0(\Omega_2)$ and so is a valid test function.
    
    Assuming this form for $\varphi$, we bound the integral in \eqref{varform} by
    \begin{align}
    \int_{ \Omega_2} &\left(c_n|\nabla_g \varphi|^2 + S_g\varphi^2\right) dV_g \leq  \int_{ \Omega_2} \left(c_n(\varphi^\prime)^2 -n(n-1)\varphi^2\right) dV_g\nonumber\\ 
    &\qquad \leq -n(n-1)Vol_g(\Omega_1) + \int_{ \Omega_2\setminus \Omega_1} \left(c_n(\tilde{\varphi}^\prime)^2 -n(n-1)\tilde{\varphi}^2\right) dV_g\nonumber\\
    &\qquad \leq -n(n-1)Vol_g(\Omega_1) + Vol_g(\Omega_2\setminus \Omega_1) \sup_{0\leq r \leq R} \left(c_n(\tilde{\varphi}^\prime(r))^2 -n(n-1)\tilde{\varphi}^2(r)\right)\nonumber\\
    &\qquad = -n(n-1)Vol_g(\Omega_1) + Vol_g(\Omega_2\setminus \Omega_1)H(\tilde{\varphi}) \label{raylineq}
    \end{align}
    where we used additionally that $S_g \leq - n(n-1)$ on $\Omega_2$. In particular, we have:
   
   \begin{lemma}\label{evlemma}
   	Let $(M,g)$ be a Riemannian manifold and suppose that there exist two open sets $\Omega_1\subset \Omega_2$ with $C^1$ boundary which satisfy \eqref{eqspace}, that the scalar curvature satisfies $S_g \leq -n(n-1)$ on $\Omega_2$, and that
   	\begin{equation}\label{dccondexisth0}
   	H_0\frac{Vol_g(\Omega_2\setminus \Omega_1)}{Vol_g(\Omega_1)} < n(n-1),
   	\end{equation}
   	where $H_0$ is the infimum defined in \eqref{snmin}. Then, the conformal Laplacian $-c_n\Delta_g + S_g$ for $(M, g)$ has a negative first eigenvalue on $\Omega_2$.
   \end{lemma}

\begin{proof}
The lemma follows from \eqref{snmin} and \eqref{raylineq}.
\end{proof}

The proof of Theorem \ref{dcexistthm} will use the explicit form of the solution of the minimisation problem \eqref{snmin} which we will obtain using \cite{MR196551}.\footnote{In fact, the work of Aronsson addressed a broader class of sup-norm minimisation problems and for rougher admissible functions.}

\begin{lemma}\label{h0lemma}
	The minimisation problem \eqref{snmin} is uniquely solved by 
	\begin{multline}\label{explmin}
		H_0 = n(n-1)\csch^2\left(\frac{\sqrt{n(n-2)}}{2} R\right) \text{ and }\\
		\tilde{\varphi}(r) = \sinh\left(\frac{\sqrt{n(n-2)}}{2} r\right)\csch\left(\frac{\sqrt{n(n-2)}}{2} R\right).
	\end{multline}
\end{lemma}

\begin{proof}
	
	Let $\mathcal{A}$ denote the set of admissible functions for \eqref{snmin}, i.e. $\psi \in \mathcal{A}$ if $\psi\in C^1([0,R])$, $\psi(0) = 0$ and $\psi(R) = 1$. Following \cite{MR196551}, we show that if $\psi\in\mathcal{A}$, $\psi' > 0$ in $(0, R)$ and $\psi$ solves the ODE $F(\psi, \psi') \equiv \text{constant}$ in $(0, R)$, then $\psi$ is the unique minimiser of \eqref{snmin}. To this end, we show that, for any function $\psi_2\in\mathcal{A}$ different from $\psi$, there exist points $\xi, \xi_2 \in (0, R)$ such that $\psi_2(\xi) = \psi(\xi_2)$ and $\psi_2'(\xi) > \psi'(\xi_2) > 0$. This will suffice to show that $\psi$ is the unique minimiser as $$H(\psi_2) \geq F(\psi_2(\xi), \psi_2'(\xi)) > F(\psi(\xi_2), \psi'(\xi_2)) = \text{ constant } = H(\psi),$$
	where the second inequality uses the fact that $\partial_z F > 0$ for $z > 0$.
	
	As $\psi' > 0$ in $(0, R)$, we may define the inverse $\alpha(y) = \psi^{-1}(y)$. Clearly $\alpha([0, 1]) = [0,R]$ and $\alpha\in C([0, 1])\cap C^1((0, 1))$. Define $r_1 = \sup\{r : r\in[0, R], \psi_2(r) \leq 0\}$ and $r_2 = \inf\{r : r \in (r_1, R], \psi_2(r) \geq 1\}$. Then $0 \leq \psi_2 \leq 1$ in $[r_1, r_2]$, $0 < \psi_2 < 1$ in $(r_1, r_2)$ and we may define the function $g := \alpha(\psi_2)$ on $[r_1, r_2]$ where $g\in C([r_1, r_2])\cap C^1((r_1, r_2))$. 
	 
	 We claim that there exists $r_0\in (r_1, r_2)$ such that $\psi(r_0) \neq \psi_2(r_0)$. Suppose not, then $\psi(r) = \psi_2(r)$ for $r\in[r_1, r_2]$ and so $\psi(r_1) = 0$ and $\psi(r_2) = 1$. By the strict monotonicity of $\psi$, we then have $r_1 = 0$ and $r_2 = R$, which implies $\psi \equiv \psi_2$, a contradiction. 
	 
	 Note that $g(r_1) = 0 \leq r_1$, $g(r_0) \neq r_0$ and $g(r_2) = R \geq r_2$. As $g\in C([r_1, r_2])\cap C^1((r_1, r_2))$, we may apply the mean value theorem to either interval $(r_1, r_0)$ or $(r_0, r_2)$ to deduce the existence of $\xi \in (r_1, r_2)$ such that $g'(\xi) > 1$. Set $\xi_2 = \alpha(\psi_2(\xi)) \in (0, R)$ and observe that $g'(\xi) = \frac{\psi_2'(\xi)}{\psi'(\xi_2)} > 1$. Consequently, $\psi_2(\xi) = \psi(\xi_2)$ and $\psi_2'(\xi) > \psi'(\xi_2) > 0$, as desired.
	 
	Direct computation verifies that the function $\tilde{\varphi}$ from the statement of the lemma solves $$F(\tilde{\varphi}, \tilde{\varphi}') = n(n-1)\csch^2\left(\frac{\sqrt{n(n-2)}}{2} R\right)$$
	and satisfies $\tilde{\varphi}(0) = 0$, $\tilde{\varphi}(R) = 1$ and $\tilde{\varphi}' > 0$ in $(0, R)$. The conclusion follows.
\end{proof}

We may now use the explicit minimiser above with the ideas from Lemma \ref{evlemma} to prove the following bounds on the first eigenvalue:

\begin{lemma}\label{lambdabounds}
	Let $(M,g)$ be a Riemannian manifold and suppose that there exist two open sets $\Omega_1\subset \Omega_2$ with $C^1$ boundary which satisfy \eqref{eqspace} and that the scalar curvature satisfies $S_g \leq -n(n-1)$ on $\Omega_2$.
	If \eqref{dccondexist} holds, then
	\begin{equation}\label{lambda1}
	\lambda < n(n-1)\frac{Vol_g(\Omega_2\setminus \Omega_1)\csch^2\left(\frac{\sqrt{n(n-2)}}{2} R\right) - Vol_g(\Omega_1)}{Vol_g(\Omega_2)} \leq 0.
	\end{equation}
	On the other hand, if \eqref{dccondexist} does not hold, we have
	\begin{equation}\label{lambda2}
	\lambda < n(n-1)\frac{Vol_g(\Omega_2\setminus \Omega_1)\csch^2\left(\frac{\sqrt{n(n-2)}}{2} R\right) - Vol_g(\Omega_1)}{Vol_g(\Omega_1)}.
	\end{equation}
	
\end{lemma}
\begin{proof}
	
	Define $\varphi$ to be the function of the form \eqref{phiform} where $\tilde{\varphi}$ is taken to be the minimiser obtained in Lemma \ref{h0lemma}. Substituting $\varphi$ in the integral in \eqref{varform} and using \eqref{raylineq} we obtain
	\begin{equation}\label{lambdaineq}
		\lambda \leq n(n-1)\frac{Vol_g(\Omega_2\setminus \Omega_1)\csch^2\left(\frac{\sqrt{n(n-2)}}{2} R\right) - Vol_g(\Omega_1)}{\|\varphi\|_{L^2(\Omega_2)}} \ .
	\end{equation}
	If the inequality above were saturated, then $\varphi$ (up to a harmless normalisation) would be a minimiser of \eqref{varform} and hence would be smooth in $\Omega_2$, contradicting the gradient discontinuity at the boundary $\partial\Omega_1$ observed using the explicit form for $\tilde{\varphi}$ obtained in Lemma \ref{h0lemma} and recalling \eqref{phiform}. Consequently, inequality \eqref{lambdaineq} above is, in fact, strict.
	
	In the case that \eqref{dccondexist} holds, the numerator in \eqref{lambdaineq} is non-positive and so, using that $\|\varphi\|_{L^2(\Omega_2)} < Vol_g(\Omega_2)$, we obtain \eqref{lambda1}. Likewise, if \eqref{dccondexist} does not hold, the numerator is positive and we use that $\|\varphi\|_{L^2(\Omega_2)} > Vol_g(\Omega_1)$ to obtain \eqref{lambda2}.
\end{proof}

	We may now prove Theorem \ref{dcexistthm} as an immediate consequence of the lemma above:
	
	\begin{proof}[Proof of Theorem \ref{dcexistthm}]
		
		Combined with assumption \eqref{asymneg}, the existence of a solution to the Yamabe problem now follows from \cite[Theorem C]{aviles1988} provided that we can show that the first eigenvalue $\lambda$ of the conformal Laplacian  on $\Omega_2$ is negative. This negativity is an immediate consequence of assumption \eqref{dccondexist} and estimate \eqref{lambda1} from Lemma \ref{lambdabounds}.
	\end{proof}

We conclude the section with the following corollaries of Theorem \ref{dcexistthm} for geodesic balls and for annuli which will aid in our discussion regarding warped product and multiply warped product type manifolds in the next subsection. In particular, for geodesic balls we have:
\begin{labelledcorollary}
	Let $(M,g)$ be a Riemannian manifold and suppose there exist constants $\alpha, R >0$ and some geodesic ball $B_{(1+\alpha)R}$ which satisfies
	\begin{equation}
	\frac{Vol_g(B_{(1+\alpha^{-1})R}\setminus B_{\alpha^{-1}R})}{Vol_g(B_{\alpha^{-1}R})} \leq \sinh^2\left(\frac{ R\sqrt{n(n-2)}}{2}\right)
	\end{equation}
	and on which the scalar curvature satisfies $S_g \leq -n(n-1)$. Then, the conformal Laplacian $-c_n\Delta_g + S_g$ for $(M, g)$ has a negative first eigenvalue on $B_{(1+\alpha^{-1})R}$.
\end{labelledcorollary}

In the case of multiply warped product metrics (see \eqref{mwp}), which have a radial fibre whose coordinate we denote by $r$, we fix some value $r_0\in \R$ and define the annular region $A_R(r_0) = \{x\in M : \left|r(x)-r_0\right| \leq R\}.$ We may then obtain as a corollary of Theorem \ref{dcexistthm},

\begin{labelledcorollary}\label{thmBp}
	Let $(M,g)$ be a Riemannian manifold and suppose there exist constants $\alpha, R >0$ and some annulus $A_{(1+\alpha)R}(r_0)$ which satisfies
	\begin{equation}\label{dcboundp}
	\frac{Vol_g(A_{(1+\alpha^{-1})R}(r_0)\setminus A_{\alpha^{-1}R}(r_0))}{Vol_g(A_{\alpha^{-1}R}(r_0))} \leq \sinh^2\left(\frac{R\sqrt{n(n-2)}}{2}\right)
	\end{equation}
	and on which the scalar curvature satisfies $S_g \leq -n(n-1)$. Then, the conformal Laplacian $-c_n\Delta_g + S_g$ for $(M, g)$ has a negative first eigenvalue on $A_{(1+\alpha^{-1})R}(r_0)$.
\end{labelledcorollary}

The proofs of the two corollaries follow directly from Theorem \ref{dcexistthm} with the appropriate choices made for $\Omega_1$, $\Omega_2$ and $R$.

\subsection{Sharpness of the eigenvalue estimate}\label{examples}

In this section we demonstrate that the volume ratio condition \eqref{dccondexist} in Theorem \ref{dcexistthm} is sharp for the negativity of the first eigenvalue. In particular, we will show:
\begin{prop}\label{sharpprop}
	Let $\beta > 0$. There exists a constant $C > 0$, such that for any large $R$ there is a complete, non-compact manifold $(M, g)$ and bounded domains $\Omega_1 \subset \Omega_2$ satisfying \eqref{eqspace} with volume ratio 
		\begin{equation}
		\frac{Vol_g(\Omega_2\setminus \Omega_1)}{Vol_g(\Omega_1)} < Ce^{\beta R}
		\label{Eq:VRExp}
		\end{equation} and with scalar curvature satisfying $S_g \leq -n(n-1)$ on $M$ for which the first eigenvalue for the conformal Laplacian on $\Omega_2$ satisfies
	\begin{equation}
		\lambda > \begin{cases}
		\frac{1}{(n-1)(n-2)}\beta^2 &\text{ if } \beta > n-1,\\
		\frac{n-1}{n-2}\beta^2 - n(n-1) &\text{ if } \beta \leq n - 1.
		\end{cases}
	\end{equation}
	In particular, if $\beta > \sqrt{n(n-2)}$, then $\lambda$ is positive.
\end{prop}

Note that, for large $R$, \eqref{Eq:VRExp} implies \eqref{Eq:VRSinh} with a slightly larger $\beta$. Remark \ref{sharpremark} follows.

We consider the product manifolds $M =\R\times N_1 \times \ldots \times N_m$ where $m\geq 1$ and each $N_i$ is a compact manifold of dimension $n_i\geq 1$ with $\sum_i n_i = n - 1$. We endow $M$ with a (multiply) warped metric $g$ of the form
\begin{equation}\label{mwp}
g = dr^2 + \sum_i p_i^2(r)h_i
\end{equation}
where $p_i:\R \rightarrow (0,\infty)$ are warping functions and  $h_i$ are metrics on each of the $N_i$. A computation shows that $(M,g)$ has Laplacian $\Delta_g = \partial_{rr} + \sum_i n_i \frac{p_i^\prime}{p_i} \partial_r + \sum_i \frac{1}{p^2_i} \Delta_{h_i}$ and 
\begin{equation}\label{scal}
S_g = -2\sum_i n_i \frac{p_i^{\prime\prime}}{p_i} - \sum_i n_i(n_i-1)\left(\frac{p_i^\prime}{p_i}\right)^2 - 2\sum_{i<j} n_i n_j\frac{p_i^\prime p_j^\prime}{p_ip_j} +\sum_i \frac{S_{h_i}}{p^2_i} \ .
\end{equation}

A particularly convenient family of warped product metrics for our purposes will be:
\begin{exmp}\label{expexmp}
	Let $(M,g)$ be a multiply warped product as in \eqref{mwp} with each $N_i = \mathbb{S}^1$ (so that $m = n-1$) and $p_i = e^{\alpha_i r}$ for some $\alpha_i\in\R$. We note that
	\begin{equation*}
	S_g = -2\left(\sum_i \alpha_i\right)^2 + 2\sum_{i<j} \alpha_i\alpha_j.
	\end{equation*}
	As a consequence, for any given $\beta$ and constant $C \geq -\frac{n}{n-1}\beta^2$, there exists a choice of $\alpha_i$ such that $\beta = \sum_i \alpha_i$ and $S_g \equiv -C$. In particular, if $|\beta|\leq n-1$, then there exists a choice of $\alpha_i$ with $\sum_i \alpha_i = \beta$ so that $S_g \equiv -n(n-1)$.
\end{exmp} 

We now turn to a study of the volume ratios of concentric balls in the example above.

\begin{lemma}\label{vrlemma}
	Let $(M,g)$ be as in Example \ref{expexmp}. Then there exists a constant $C$ depending only on $\alpha_i$ such that, for any $R$ sufficiently large, one can find concentric balls $B_{2R}$ and $B_R$ such that
	\begin{equation*}
	\frac{Vol_{g}(B_{2R}\setminus B_R)}{Vol_{g}(B_{R})} \leq Ce^{\beta R}.
	\end{equation*}
\end{lemma}

\begin{proof}

Fixing some large $R$ and taking some ball $B_R(p_0)$, writing $p_0 = (r_0, x_0)$, we have immediately that $$B_R(p_0) \subset \{p : |r(p) - r_0| \leq R\}.$$
On the other hand, defining $S_0 := \{r_0\}\times \mathbb{T}^{n-1}$ we have that $$ d_g(p,p_0) \leq d_g(p,S_0) + \text{diam}(S_0) = r(p)-r_0 + e^{r_0}\text{diam}(\mathbb{T}^{n-1}).$$
Consequently, writing $A_0 = e^{r_0}\text{diam}(\mathbb{T}^{n-1})$, we have that $$\{|r-r_0| \leq R - A_0\} \subset B_R(p_0) \subset \{|r-r_0|\leq R\}$$
from which we obtain
$$Vol_g(B_{2R}(p_0)) \leq \int_{\mathbb{T}^{n-1}} \int_{r_0 - 2R}^{r_0 + 2R}e^{\beta r} dr \ dx = \beta Vol(\mathbb{T}^{n-1})\left(e^{\beta(r_0 + 2R)} - e^{\beta(r_0 - 2R)}\right)$$
and
\begin{align*}
Vol_g(B_R(p_0)) &\geq \int_{\mathbb{T}^{n-1}} \int_{r_0 - (R - A_0)}^{r_0 + (R - A_0)}e^{\beta r} dr \ dx 
= \beta Vol(\mathbb{T}^{n-1})e^{-A_0}\left(e^{\beta(r_0 + R)} - e^{\beta(r_0 - R + 2A_0)}\right)
\end{align*}
from which we may estimate
\begin{equation*}
\frac{Vol_g(B_{2R}\setminus B_R)}{Vol_g(B_R)} \leq e^{A_0}\frac{e^{\beta(r_0 + 2R)} - e^{\beta(r_0 - 2R)}}{e^{\beta(r_0 + R)} - e^{\beta(r_0 - R + 2A_0)}} - 1.
\end{equation*}
Consequently, $(M, g)$ satisfies the following bound on the volume ratio for large balls
\begin{equation}
\frac{Vol_g(B_{2R}\setminus B_R)}{Vol_g(B_R)} \leq Ce^{\beta R}\left(1 + \mathcal{O}(e^{-2(n-1)R})\right) - 1
\end{equation}
and so, taking $R$ large, we obtain
\begin{equation*}
\frac{Vol_{g}(B_{2R}\setminus B_R)}{Vol_{g}(B_{R})} \leq Ce^{\beta R}.
\end{equation*}
for some large constant $C>0$ depending only on $\alpha_i$. 
\end{proof}

We are now ready to prove Proposition \ref{sharpprop}. However, before this, we make the following remark:

\begin{remark}
	As a consequence of Lemma \ref{vrlemma}, we may use Corollary \ref{thmBp} of Theorem \ref{dcexistthm} to prove that any manifold of the form of Example \ref{expexmp} admits a solution to the Yamabe problem, provided that $\beta < \sqrt{n(n-2)}$.
\end{remark}

\begin{proof}[Proof of Proposition \ref{sharpprop}]
	Let $(M, g)$ be of the form of Example \ref{expexmp} where $\alpha_i$ are chosen such that $\sum_i \alpha_i = \beta$ and
	\begin{equation*}
		S_g \equiv \begin{cases}
		-\frac{n}{n-1}\beta^2 \text{ if } \beta > n-1,\\
		-n(n-1) \text{ if } \beta \leq n - 1.
		\end{cases}
	\end{equation*}
	In particular, $S_g \leq -n(n-1)$. We note that, by the same argument as in the proof of Lemma \ref{vrlemma}, there exists a $C(\alpha_1, \ldots, \alpha_{n-1}) > 0$ such that we may choose concentric annuli $\Omega_2 = A_{2R}$ and $\Omega_1 = A_R$ satisfying
	\begin{equation*}
	\frac{Vol_{g}(A_{2R}\setminus A_R)}{Vol_{g}(A_{R})} \leq C(\alpha_1, \ldots, \alpha_{n-1})e^{\beta R}.
	\end{equation*} 
	
	As the eigenspace corresponding to the first eigenvalue $\lambda$ is necessarily one dimensional, we must have that the first eigenfunctions $\varphi$ of the conformal Laplacian on $A_{2R}$ are radially symmetric (i.e. $\varphi = \varphi(r)$) by virtue of the symmetry of the torus. Consequently, $\varphi$ satisfies the constant coefficient ODE
	\begin{equation*}
	\varphi^{\prime\prime} + \beta\varphi^\prime + \frac{\lambda - S_g}{c_n}\varphi = 0
	\end{equation*}
	on $\Omega_2$ subjected to the zero Dirichlet boundary condition.
	
	It follows that the corresponding characteristic equation must have complex, non-real roots, which implies that
	\begin{equation*}
	\beta^2 - \frac{{(n-2)(\lambda - S_g)}}{(n-1)} < 0.
	\end{equation*}
	Consequently, $\lambda > \frac{n-1}{n-2}\beta^2 - S_g$ and so, recalling our choice of $S_g$, the proposition follows.
\end{proof} 

To conclude, we remark that one may use Theorem \ref{dcexistthm} to show existence of solutions to the Yamabe problem for manifolds obtained from the warped product type metrics discussed above via perturbations which preserve the volume ratio and scalar curvature conditions. 

In a future work, it would be interesting to explore which results similar to those achieved for ALH manifolds could be obtained in the setting of Theorem \ref{dcexistthm} and to understand to what extent the upper bound on the scalar curvature could be loosened.


\begin{thebibliography}{CGNP18}
	
	\bibitem[ACF92]{MR1186044}
	L. Andersson, P.~T. Chru\'{s}ciel, and H. Friedrich.
	\newblock On the regularity of solutions to the {Y}amabe equation and the
	existence of smooth hyperboloidal initial data for {E}instein's field
	equations.
	\newblock {\em Comm. Math. Phys.}, 149(3):587--612, 1992.
	
	\bibitem[AILA18]{MR3761652}
	P.~T. Allen, J. Isenberg, J.~M. Lee, and I. Stavrov Allen.
	\newblock Weakly asymptotically hyperbolic manifolds.
	\newblock {\em Comm. Anal. Geom.}, 26(1):1--61, 2018.
	
	\bibitem[AM85]{MR816672}
	P. Aviles and R.~C. McOwen.
	\newblock Conformal deformations of complete manifolds with negative curvature.
	\newblock {\em J. Differential Geom.}, 21(2):269--281, 1985.
	
	\bibitem[AM88a]{MR932852}
	P. Aviles and R.~C. McOwen.
	\newblock Complete conformal metrics with negative scalar curvature in compact
	{R}iemannian manifolds.
	\newblock {\em Duke Math. J.}, 56(2):395--398, 1988.
	
	\bibitem[AM88b]{aviles1988}
	P. Aviles and R.~C. McOwen.
	\newblock Conformal deformation to constant negative scalar curvature on
	noncompact {R}iemannian manifolds.
	\newblock {\em J. Differential Geom.}, 27(2):225--239, 1988.
	
	\bibitem[Aro65]{MR196551}
	G. Aronsson.
	\newblock Minimization problems for the functional {${\rm
			sup}_{x}\,F(x,\,f(x),\,f^{\prime} (x))$}.
	\newblock {\em Ark. Mat.}, 6:33--53 (1965), 1965.
	
	\bibitem[Aub76]{MR431287}
	T. Aubin.
	\newblock \'{E}quations diff\'{e}rentielles non lin\'{e}aires et probl\`eme de
	{Y}amabe concernant la courbure scalaire.
	\newblock {\em J. Math. Pures Appl. (9)}, 55(3):269--296, 1976.
	
	\bibitem[Avi82]{MR660747}
	P. Aviles.
	\newblock A study of the singularities of solutions of a class of nonlinear
	elliptic partial differential equations.
	\newblock {\em Comm. Partial Differential Equations}, 7(6):609--643, 1982.
	
	\bibitem[BG11]{MR2836592}
	E. Bahuaud and R. Gicquaud.
	\newblock Conformal compactification of asymptotically locally hyperbolic
	metrics.
	\newblock {\em J. Geom. Anal.}, 21(4):1085--1118, 2011.
	
	\bibitem[BM09]{MR2472174}
	S. Brendle and F.~C. Marques.
	\newblock Blow-up phenomena for the {Y}amabe equation. {II}.
	\newblock {\em J. Differential Geom.}, 81(2):225--250, 2009.
	
	\bibitem[Bre08]{MR2425176}
	S. Brendle.
	\newblock Blow-up phenomena for the {Y}amabe equation.
	\newblock {\em J. Amer. Math. Soc.}, 21(4):951--979, 2008.
	
	\bibitem[CGNP18]{MR3801943}
	P.~T. Chru\'{s}ciel, G.~J. Galloway, L. Nguyen, and T.~T. Paetz.
	\newblock On the mass aspect function and positive energy theorems for
	asymptotically hyperbolic manifolds.
	\newblock {\em Classical Quantum Gravity}, 35(11):115015, 38, 2018.
	
	\bibitem[CGS89]{MR982351}
	L.~A. Caffarelli, B. Gidas, and J. Spruck.
	\newblock Asymptotic symmetry and local behavior of semilinear elliptic
	equations with critical {S}obolev growth.
	\newblock {\em Comm. Pure Appl. Math.}, 42(3):271--297, 1989.
	
	\bibitem[CH03]{MR2038048}
	P.~T. Chru\'{s}ciel and M. Herzlich.
	\newblock The mass of asymptotically hyperbolic {R}iemannian manifolds.
	\newblock {\em Pacific J. Math.}, 212(2):231--264, 2003.
	
	\bibitem[CL99]{MR1679784}
	C.-C. Chen and C.-S. Lin.
	\newblock On the asymptotic symmetry of singular solutions of the scalar
	curvature equations.
	\newblock {\em Math. Ann.}, 313(2):229--245, 1999.
	
	\bibitem[CP08]{MR2413198}
	P.~T. Chru\'{s}ciel and D. Pollack.
	\newblock Singular {Y}amabe metrics and initial data with exactly
	{K}ottler-{S}chwarzschild-de {S}itter ends.
	\newblock {\em Ann. Henri Poincar\'{e}}, 9(4):639--654, 2008.
	
	\bibitem[Fin98]{MR1641725}
	D.~L. Finn.
	\newblock Existence of positive solutions to {$\Delta_gu=u^q+Su$} with
	prescribed singularities and their geometric implications.
	\newblock {\em Comm. Partial Differential Equations}, 23(9-10):1795--1814,
	1998.
	
	\bibitem[Fin99]{MR1760721}
	D.~L. Finn.
	\newblock On the negative case of the singular {Y}amabe problem.
	\newblock {\em J. Geom. Anal.}, 9(1):73--92, 1999.
	
	\bibitem[Gic13]{MR3169748}
	R. Gicquaud.
	\newblock Conformal compactification of asymptotically locally hyperbolic
	metrics {II}: weakly {ALH} metrics.
	\newblock {\em Comm. PDEs}, 38(8):1313--1367, 2013.
	
	\bibitem[GNN79]{MR544879}
	B. Gidas, W.-M. Ni, and L. Nirenberg.
	\newblock Symmetry and related properties via the maximum principle.
	\newblock {\em Comm. Math. Phys.}, 68(3):209--243, 1979.
	
	\bibitem[GNN81]{MR634248}
	B. Gidas, W.-M. Ni, and L. Nirenberg.
	\newblock Symmetry of positive solutions of nonlinear elliptic equations in
	{${\bf R}^{n}$}.
	\newblock In {\em Mathematical analysis and applications, {P}art {A}}, volume~7
	of {\em Adv. in Math. Suppl. Stud.}, pages 369--402. Academic Press, New
	York-London, 1981.
	
	\bibitem[Gra17]{grahampaper}
	C.~R. Graham.
	\newblock Volume renormalization for singular {Y}amabe metrics.
	\newblock {\em Proc. Amer. Math. Soc.}, 145(4):1781--1792, 2017.
	
	\bibitem[GT01]{gilbarg2001elliptic}
	D. Gilbarg and N.~S. Trudinger.
	\newblock {\em Elliptic Partial Differential Equations of Second Order}.
	\newblock Classics in Mathematics. U.S. Government Printing Office, 2001.
	
	\bibitem[GW17]{MR3668619}
	A.~R. Gover and A. Waldron.
	\newblock Renormalized volume.
	\newblock {\em Comm. Math. Phys.}, 354(3):1205--1244, 2017.
	
	\bibitem[HL21]{han2019singular}
	Q. Han and Y. Li.
	\newblock Singular solutions to the {Y}amabe equation with prescribed
	asymptotics.
	\newblock {\em J. Differential Equations}, 274:127--150, 2021.
	
	\bibitem[Hog20]{thesis}
	J. Hogg.
	\newblock {\em The Yamabe Problem on Non-Compact Manifolds of Negative
		Curvature Type}.
	\newblock PhD thesis, University of Oxford, 2020.
	
	\bibitem[HS20]{MR4096721}
	Q. Han and W.-M. Shen.
	\newblock The {L}oewner-{N}irenberg problem in singular domains.
	\newblock {\em J. Funct. Anal.}, 279(6):108604, 43, 2020.
	
	\bibitem[Jia21]{MR4269604}
	X. Jiang.
	\newblock Boundary expansion for the {L}oewner-{N}irenberg problem in domains
	with conic singularities.
	\newblock {\em J. Funct. Anal.}, 281(7):Paper No. 109122, 41, 2021.
	
	\bibitem[Jin88]{MR1032773}
	Z.-R. Jin.
	\newblock A counterexample to the {Y}amabe problem for complete noncompact
	manifolds.
	\newblock In {\em Partial differential equations ({T}ianjin, 1986)}, volume
	1306 of {\em Lecture Notes in Math.}, pages 93--101. Springer, Berlin, 1988.
	
	\bibitem[Kel57]{https://doi.org/10.1002/cpa.3160100402}
	J.~B. Keller.
	\newblock {On solutions of $\Delta u=f(u)$}.
	\newblock {\em Comm. Pure Appl. Math.}, 10(4):503--510,
	1957.
	
	\bibitem[KMPS99]{MR1666838}
	N. Korevaar, R. Mazzeo, F. Pacard, and R. Schoen.
	\newblock Refined asymptotics for constant scalar curvature metrics with
	isolated singularities.
	\newblock {\em Invent. Math.}, 135(2):233--272, 1999.
	
	\bibitem[KMS09]{MR2477893}
	M.~A. Khuri, F.~C. Marques, and R. Schoen.
	\newblock A compactness theorem for the {Y}amabe problem.
	\newblock {\em J. Differential Geom.}, 81(1):143--196, 2009.
	
	\bibitem[Lab05]{2005Labutin}
	D.~A. Labutin.
	\newblock Thinness for scalar-negative singular {Y}amabe metrics.
	\newblock {\em Pre-print, arXiv:math/0506226}, 2005.
	
	\bibitem[LN74]{MR0358078}
	C. Loewner and L. Nirenberg.
	\newblock Partial differential equations invariant under conformal or
	projective transformations.
	\newblock In {\em Contributions to analysis (a collection of papers dedicated
		to {L}ipman {B}ers)}, pages 245--272. Academic Press, New York, 1974.
	
	\bibitem[LZ04]{MR2057491}
	Y.~Y. Li and L. Zhang.
	\newblock A {H}arnack type inequality for the {Y}amabe equation in low
	dimensions.
	\newblock {\em Calc. Var. PDEs}, 20(2):133--151,
	2004.
	
	\bibitem[LZ05]{MR2164927}
	Y.~Y. Li and Lei Zhang.
	\newblock Compactness of solutions to the {Y}amabe problem. {II}.
	\newblock {\em Calc. Var. PDEs}, 24(2):185--237,
	2005.
	
	\bibitem[LZ07]{MR2309836}
	Y.~Y. Li and L. Zhang.
	\newblock Compactness of solutions to the {Y}amabe problem. {III}.
	\newblock {\em J. Funct. Anal.}, 245(2):438--474, 2007.
	
	\bibitem[Mar08]{MR2393072}
	F.~C. Marques.
	\newblock Isolated singularities of solutions to the {Y}amabe equation.
	\newblock {\em Calc. Var. PDEs}, 32(3):349--371,
	2008.
	
	\bibitem[Maz91]{MR1142715}
	R. Mazzeo.
	\newblock Regularity for the singular {Y}amabe problem.
	\newblock {\em Indiana Univ. Math. J.}, 40(4):1277--1299, 1991.
	
	\bibitem[MP96]{MR1425579}
	R. Mazzeo and F. Pacard.
	\newblock A construction of singular solutions for a semilinear elliptic
	equation using asymptotic analysis.
	\newblock {\em J. Differential Geom.}, 44(2):331--370, 1996.
	
	\bibitem[Oss57]{pjm/1103043236}
	R. Osserman.
	\newblock {On the inequality $\Delta u\geq f(u)$.}
	\newblock {\em Pacific J. Math.}, 7(4):1641 -- 1647, 1957.
	
	\bibitem[RRV97]{MR1465899}
	A. Ratto, M. Rigoli, and L. Veron.
	\newblock Scalar curvature and conformal deformations of noncompact
	{R}iemannian manifolds.
	\newblock {\em Math. Z.}, 225(3):395--426, 1997.
	
	\bibitem[Sch84]{MR788292}
	R. Schoen.
	\newblock Conformal deformation of a {R}iemannian metric to constant scalar
	curvature.
	\newblock {\em J. Differential Geom.}, 20(2):479--495, 1984.
	
	\bibitem[SY88]{MR931204}
	R. Schoen and S.-T. Yau.
	\newblock Conformally flat manifolds, {K}leinian groups and scalar curvature.
	\newblock {\em Invent. Math.}, 92(1):47--71, 1988.
	
	\bibitem[Tru68]{MR240748}
	N.~S. Trudinger.
	\newblock Remarks concerning the conformal deformation of {R}iemannian
	structures on compact manifolds.
	\newblock {\em Ann. Scuola Norm. Sup. Pisa Cl. Sci. (3)}, 22:265--274, 1968.
	
	\bibitem[V{\'e}r81]{MR626622}
	L. V{\'e}ron.
	\newblock Singularit\'{e}s \'{e}liminables d'\'{e}quations elliptiques non
	lin\'{e}aires.
	\newblock {\em J. Differential Equations}, 41(1):87--95, 1981.
	
	\bibitem[Wan01]{MR1879228}
	X. Wang.
	\newblock The mass of asymptotically hyperbolic manifolds.
	\newblock {\em J. Differential Geom.}, 57(2):273--299, 2001.
	
	\bibitem[XZ22]{MR4436213}
	J. Xiong and L. Zhang.
	\newblock Isolated singularities of solutions to the {Y}amabe equation in
	dimension 6.
	\newblock {\em Int. Math. Res. Not. IMRN}, (12):9571--9597, 2022.
	
	\bibitem[Yam60]{MR125546}
	H. Yamabe.
	\newblock On a deformation of {R}iemannian structures on compact manifolds.
	\newblock {\em Osaka Math. J.}, 12:21--37, 1960.
	
\end{thebibliography}
\end{document}